\documentclass{amsproc}
\usepackage{upref}
\numberwithin{equation}{section}

\theoremstyle{plain}
\newtheorem{theorem}{Theorem}[section]
\newtheorem{proposition}[theorem]{Proposition}
\newtheorem{corollary}[theorem]{Corollary}
\newtheorem{definition}[theorem]{Definition}

\theoremstyle{definition}
\newtheorem{remark}[theorem]{Remark}

\begin{document}
\newcommand{\eq}{equation}
\newcommand{\real}{\ensuremath{\mathbb R}}
\newcommand{\rn}{\ensuremath{{\mathbb R}^n}}
\newcommand{\rp}{\ensuremath{{\mathbb R}^2}}
\newcommand{\nat}{\ensuremath{\mathbb N}}
\newcommand{\no}{\ensuremath{\nat_0}}
\newcommand{\ganz}{\ensuremath{{\mathbb Z}}}
\newcommand{\zn}{\ensuremath{\ganz^n}}
\newcommand{\comp}{\ensuremath{{\mathbb C}}}
\newcommand{\Pbm}{\ensuremath{{\mathbb P}}}
\newcommand{\hra}{\mbox{$ \ \hookrightarrow \ $}}
\newcommand{\di}{\ensuremath{{\rm d}}}
\newcommand{\id}{\mbox{\rm id}}
\newcommand{\supp}{\mbox{\rm supp }}
\newcommand{\ve}{\ensuremath{\varepsilon}}
\newcommand{\vp}{\ensuremath{\varphi}}
\newcommand{\vk}{\ensuremath{{ \kappa}}}
\newcommand{\vr}{\ensuremath{\varrho}}
\newcommand{\As}{\ensuremath{A^s_{p,q}}}
\newcommand{\Bs}{\mbox{$B^s_{p,q}$}}
\newcommand{\Fs}{\mbox{$F^s_{p,q}$}}
\newcommand{\Bt}{\mbox{$\widetilde{B}^s_{p,q}$}}
\newcommand{\Bc}{\mbox{$\overset{\, \circ}{B}{}^s_{p,q}$}}
\newcommand{\pa}{\ensuremath{\partial}}
\newcommand{\pr}{\pageref}
\newcommand{\Om}{\ensuremath{\Omega}}
\newcommand{\ext}{\mbox{\rm ext}}
\newcommand{\re}{\ensuremath{\mathrm{re}}}
\newcommand{\wtw}{\ensuremath{\widetilde{w}_{-\alpha}}}
\newcommand{\os}{\ensuremath{\overset}}
\newcommand{\ol}{\ensuremath{\overline}}
\newcommand{\Li}{\ensuremath{\overset{\circ}{L}}}
\newcommand{\Ai}{\ensuremath{\os{\, \circ}{A}}}
\newcommand{\Aa}{\ensuremath{\os{\, \ast}{A}}}
\newcommand{\Ba}{\ensuremath{\os{\, \ast}{B}}}
\newcommand{\Fa}{\ensuremath{\os{\, \ast}{F}}}
\newcommand{\ba}{\ensuremath{\os{\, \ast}{b}_{p,q}}}
\newcommand{\bb}{\ensuremath{\os{\, \ast}{b}_{1,1}}}
\newcommand{\bc}{\ensuremath{\os{\, \ast}{b}_{\infty, \infty}}}
\newcommand{\bd}{\ensuremath{\os{\, \ast}{b}_{p,p}}}
\newcommand{\fa}{\ensuremath{\os{\, \ast}{f}_{p,q}}}
\newcommand{\Aas}{\ensuremath{\Aa{}^s_{p,q}}}
\newcommand{\Bas}{\ensuremath{\Ba{}^s_{p,q}}}
\newcommand{\Fas}{\ensuremath{\Fa{}^s_{p,q}}}
\newcommand{\Ca}{\ensuremath{\os{\, \ast}{{\mathcal C}}}}
\newcommand{\Wa}{\ensuremath{\os{\, \ast}{W}}}
\newcommand{\Cas}{\ensuremath{\Ca{}^s}}
\newcommand{\Cc}{\ensuremath{{\mathcal C}}}
\newcommand{\dom}{\ensuremath{\mathrm{dom \,}}}
\newcommand{\bmo}{\ensuremath{\mathrm{bmo}}}
\newcommand{\BMO}{\ensuremath{\mathrm{BMO}}}
\newcommand{\tr}{\ensuremath{\mathrm{tr}}}
\newcommand{\Div}{\ensuremath{\mathrm{div \,}}}
\newcommand{\Pbb}{\ensuremath{\mathbb P}}
\newcommand{\wh}{\ensuremath{\widehat}}
\newcommand{\wt}{\ensuremath{\widetilde}}
\newcommand{\Lloc}{\ensuremath{L_1^{\mathrm{loc}}}}
\newcommand{\Sreg}{\ensuremath{S'(\rn)^{\mathrm{reg}}}}
\newcommand{\Real}{\ensuremath{{\mathrm{Re}}}}

\title{Tempered homogeneous function spaces, II}
\author{Hans Triebel}
\address{Institut f\"{u}r Mathematik, Fakult\"{a}t f\"{u}r Mathematik und Informatik, Friedrich-Schiller-Universit\"{a}t Jena, 07737 Jena, Germany}
\email{hans.triebel@uni-jena.de}
\subjclass[2010]{46E35, 42B35}
\keywords{Homogeneous function spaces}
\date{March, 7, 2016}

\begin{abstract}
This paper deals with homogeneous function spaces of Besov-Sobolev type within the framework of tempered distributions in Euclidean $n$-space based on
Gauss-Weierstrass semi-groups. Related Fourier-analytical descriptions are incorporated afterwards as so-called domestic norms. This approach avoids
the usual ambiguity modulo polynomials when homogeneous function spaces are considered in the context of homogeneous tempered distributions. 
The motivation to deal with these spaces comes from (nonlinear) heat and Navier-Stokes equations, but also from Keller-Segel systems and other PDE models of chemotaxis.
\end{abstract}

\maketitle

\setcounter{tocdepth}{2}

\tableofcontents

\section{Introduction and motivation}   \label{S1}
This paper is a complementing survey to \cite{T15}. We try to be selfcontained repeating some definitions and basic assertions. Let $\As (\rn)$, 
$n\in \nat$, $A \in \{B,F \}$ with
\begin{\eq}   \label{1.1}
0< p,q \le \infty \qquad \text{and} \qquad s\in \real
\end{\eq}
be the nowadays well-known {\em inhomogeneous function spaces} treated in the framework of the dual pairing $\big(S(\rn), S'(\rn) \big)$. The theory
of these spaces, including their history, references and special cases may be found in \cite{T83, T92, T06}. Here $S(\rn)$ is the usual Schwartz space
of infinitely differentiable rapidly decreasing functions in $\rn$ and $S'(\rn)$ is its dual, the space of tempered distributions. The related
{\em homogeneous function spaces} $\dot{A}^s_{p,q} (\rn)$ again with $A \in \{B,F \}$ and \eqref{1.1} are usually treated in the context of the dual 
pairing $\big( \dot{S} (\rn), \dot{S}' (\rn) \big)$ where
\begin{\eq}   \label{1.1a}
\dot{S} (\rn) = \big\{ \vp \in S(\rn): \ (D^\alpha \wh{\vp}) (0) =0, \ \alpha \in \nat^n_0 \big\}.
\end{\eq}
Recall that $\wh{\vp}$ is the Fourier transform of \vp.
A first brief account of these homogeneous spaces, again with references and special cases, had been given in \cite[Chapter 5]{T83}. A more detailed
elaborated and updated version of the theory of these homogeneous spaces may be found in \cite[Chapter 2]{T15}. It is well known that these homogeneous
spaces must be considered modulo polynomials. This causes some disturbing (topological) troubles if one uses homogeneous spaces in some applications,
especially in the context of nonlinear PDEs. In \cite{T15} we offered a theory of {\em tempered homogeneous  function spaces} $\Aas (\rn)$, $A \in \{B,F \}$ in the framework of the dual pairing $\big( S(\rn), S'(\rn) \big)$ where $p,q,s$ are restricted (mostly but not exclusively) to the {\em distinguished strip}
\begin{\eq}   \label{1.2}
0<p,q \le \infty, \qquad n \big( \frac{1}{p} - 1 \big) <s< \frac{n}{p}.
\end{\eq}
This avoids any ambiguity modulo polynomials and gives the possibility to transfer many properties of the inhomogeneous spaces to the corresponding
tempered homogeneous ones. One has, in addition, 
\begin{\eq}  \label{1.3}
\| f(\lambda \cdot) \, | \Aas (\rn) \| = \lambda^{s- \frac{n}{p}} \, \| f \, | \Aas (\rn) \|, \qquad 0<\lambda <\infty,
\end{\eq}
of {\em homogeneity order} $s - \frac{n}{p}$. One may think about the Lebesgue spaces $L_p (\rn) = \Fa{}^0_{p,2} (\rn) = F^0_{p,2} (\rn)$, $1<p<\infty$,
as a proto-type of these spaces. The restriction of $p,q,s$ to the indicated distinguished strip \eqref{1.2} maybe disturbing from the point of view of the theory of function spaces, but it is natural (some limiting cases will be incorporated below) and covers in particular homogeneous spaces which are
of interest in some applications, especially to a few distinguished nonlinear PDEs. In particular, our motivation to develop the theory of {\em 
tempered homogeneous spaces} $\Aas (\rn)$ originates from the Navier-Stokes equations 
\begin{align}   \label{1.4}
\pa_t  u - \Delta u + \Pbb \, \Div (u \otimes u) &=0 &&\text{in $\rn \times (0, T)$,}&& \\  \label{1.5}
u(\cdot, 0) &= u_0 &&\text{in \rn},&&
\end{align}
$n \ge 2$, $0<T \le \infty$. Here $u(x,t) = \big( u^1 (x,t), \ldots, u^n (x,t) \big)$, whereas $u \otimes u$ is the usual tensor product and $\Pbb$
is the Leray projector, a singular Calder\'{o}n-Zygmund operator of homogeneity order 0. We do not deal here with Navier-Stokes equations and refer the
reader for details, explanations and the abundant literature to \cite{BCD11, T13, T14}. Of interest for us is the following homogeneity assertion. If
$u(x,t)$ is a solution of \eqref{1.4} then $u_\lambda (x,t) = \lambda \, u(\lambda x, \lambda^2 t)$, $0<\lambda <\infty$, 
is also a solution of \eqref{1.4} now in $\rn
\times (\lambda^{-2} T)$ where one has to adapt the initial data. This fits pretty well to homogeneity assertions of type \eqref{1.2}, \eqref{1.3}. As
a consequence, spaces
\begin{\eq}   \label{1.6}
\As (\rn), \quad \dot{A}^s_{p,q} (\rn), \quad \Aas (\rn), \qquad 0<p,q \le \infty \quad \text{with} \quad s = -1 + \frac{n}{p}
\end{\eq}
are called {\em critical} (for Navier-Stokes equations), {\em supercritical} if $s> -1 + \frac{n}{p}$ and {\em subcritical} if $s<-1 + \frac{n}{p}$.
Details may be found in \cite[pp.\,1,2]{T15}. Of special interest are the critical and supercritical spaces with $-1 + \frac{n}{p} \le s <\frac{n}{p}$.
This fits in the scheme \eqref{1.3}. Our own approach to Navier-Stokes equations in \cite{T13, T14} is based on inhomogeneous supercritical spaces (and
is local in time). But most of the recent contributions to Navier-Stokes equations  in the context of Fourier analysis rely on homogeneous critical
spaces $\dot{A}^{-1 + \frac{n}{p}}_{p,q} (\rn)$. This applies also to other nonlinear PDEs originating from physics. The reader may consult \cite{BCD11}
and some more recent papers quoted in \cite{T13, T14}. An additional motivation to deal with tempered homogeneous function spaces comes from {\em
chemotaxis}, the movement of biological cells or organisms in response to chemical gradients. Detailed descriptions  including the biological background
and the model equations studied today may be found in the surveys \cite{Hor03, Hor04}, \cite{HiP09} and the recent book \cite{Per15}. In  addition to 
the classical approach (smooth functions in bounded domains) a lot of attention has been paid in recent times to study the underlying Keller-Segel
systems in the context of the above spaces $\As (\rn)$, $\dot{A}^s_{p,q} (\rn)$. The prototype of these equations is given by
\begin{align}   \label{1.7}
\pa_t u - \Delta u + \Div(u \nabla v) &=0, &&\text{$x\in \rn$, $0<t<T$}, && \\  \label{1.8}
\pa_t v - \Delta v + \alpha v &=u, &&\text{$x\in \rn$, $0<t<T$}, && \\ \label{1.9}
u(\cdot,0) &= u_0, &&\text{$x\in \rn$}, && \\  \label{1.10}
v(\cdot,0) &= v_0, &&\text{$x\in \rn$}, && 
\end{align}
where $\alpha \ge 0$ is the so-called damping constant. Here $u= u(x,t)$ and $v=v(x,t)$ are scalar functions describing the cell density and the 
concentration of the chemical signals, respectively. If $\alpha =0$ then one is in a similar position as in the case of the Navier-Stokes equations and
one can again ask which spaces $\As (\rn)$, $\dot{A}^s_{p,q} (\rn)$ should be called critical, supercritical and subcritical for Keller-Segel equations.
If $u(x,t)$, $v(x,t)$ is a solution of \eqref{1.7}, \eqref{1.8} then $u_\lambda (x,t) = \lambda^2 u(\lambda x, \lambda^2 t)$, $ v_\lambda (x,t) = 
v(\lambda x, \lambda^2 t)$ is also a solution of \eqref{1.7}, \eqref{1.8} now in $\rn \times (0, \lambda^{-2} T)$, where one has to adapt the initial 
data. Then the spaces
\begin{\eq}  \label{1.11}
\As (\rn), \quad \dot{A}^s_{p,q} (\rn), \quad \Aas (\rn), \qquad 0<p,q \le \infty, \quad s= -2 + \frac{n}{p},
\end{\eq}
are called {\em critical} (for Keller-Segel equations), {\em supercritical} if $s >-2 + \frac{n}{p}$ and {\em subcritical} if $s<-2 + \frac{n}{p}$. 
Of special interest are the critical and supercritical spaces with $-2 + \frac{n}{p} \le s \le \frac{n}{p}$. Again this seems to fit in the scheme of
\eqref{1.2}, \eqref{1.3}. Details may be found in \cite{T16}. But there are some differences compared with Navier-Stokes equations. Homogeneous spaces 
with \eqref{1.2} cover the cases of interest for Navier-Stokes equations in $\rn$ with $2\le n\in \nat$. Keller-Segel systems with $n=1$ (ordinary
nonlinear equations) attracted some attention but they are not covered by our approach in \cite{T16}. According to the literature the most interesting
case for Keller-Segel systems is $n=2$ (biological cells in so-called Petri dishes in response to chemicals). But then \eqref{1.11} (with $n=2$) suggests
to have a closer  look at (tempered) homogeneous spaces $\Aas (\rn)$, $0<p,q\le \infty$ in the limiting cases $s = n\big( \frac{1}{p} - 1 \big)$ and
$s = \frac{n}{p}$. It is one aim of this paper to complement some corresponding considerations in \cite{T15} in this direction. But on the other hand
it is our main aim to continue the study of the tempered homogeneous spaces $\Aas (\rn)$ mostly with \eqref{1.2} asking for further properties. We deal
also with some examples based on Riesz kernels demonstrating the decisive differences between the above homogeneous spaces and their (more flexible)
inhomogeneous counterparts $\As (\rn)$. 

In Section \ref{S2} we introduce the tempered homogeneous spaces $\Aas (\rn)$ following closely \cite{T15}. We repeat some basic assertions and prove
new ones (mostly in limiting situations). Section \ref{S3} deals with new properties of the spaces $\Aas (\rn)$ complementing \cite{T15}.

\section{Definitions and basic assertions}   \label{S2}
\subsection{Preliminaries and inhomogeneous spaces}   \label{S2.1}
We use standard notation. Let $\nat$ be the collection of all natural numbers and $\no = \nat \cup \{0 \}$. Let $\rn$ be Euclidean $n$-space, where
$n\in \nat$. Put $\real = \real^1$, whereas $\comp$ is the complex plane.
Let $S(\rn)$ be the Schwartz space of all complex-valued rapidly decreasing infinitely differentiable functions on $\rn$ and let $S' (\rn)$ be the space of all tempered distributions on $\rn$, the dual of $S(\rn)$. 
Let $D(\rn) = C^\infty_0 (\rn)$ be the collection of all functions $f\in S(\rn)$ with compact support in \rn. As usual $D' (\rn)$ stands for the space
of all distributions in \rn. Furthermore, $L_p (\rn)$ with $0< p \le \infty$, is the standard complex quasi-Banach space with respect to the Lebesgue measure, quasi-normed by
\begin{\eq}   \label{2.1}
\| f \, | L_p (\rn) \| = \Big( \int_{\rn} |f(x)|^p \, \di x \Big)^{1/p}
\end{\eq}
with the usual modification if $p=\infty$. 
Similarly $L_p (M)$ where $M$ is a Lebesgue-measurable subset of \rn. As usual $\ganz$ is the collection of all integers; and $\zn$ where $n\in \nat$ denotes the
lattice of all points $m= (m_1, \ldots, m_n) \in \rn$ with $m_k \in \ganz$. Let $Q_{j,m} = 2^{-j}m + 2^{-j} (0,1)^n$ with $j\in \ganz$ and $m\in \zn$ be the usual dyadic cubes in $\rn$,
$n \in \nat$, with sides of length $2^{-j}$ parallel to the axes of coordinates and $2^{-j} m$ as the lower left corner. As usual, $L_p^{\mathrm{loc}}
(\rn)$ collects all locally $p$-integrable functions $f$, that is $f\in L_p (M)$ for any bounded Lebesgue measurable set $M$ in \rn. 

If $\vp \in S(\rn)$ then
\begin{\eq}  \label{2.2}
\wh{\vp} (\xi) = (F \vp)(\xi) = (2\pi )^{-n/2} \int_{\rn} e^{-ix \xi} \vp (x) \, \di x, \qquad \xi \in  \rn,
\end{\eq}
denotes the Fourier transform of \vp. As usual, $F^{-1} \vp$ and $\vp^\vee$ stand for the inverse Fourier transform, given by the right-hand side of
\eqref{2.2} with $i$ in place of $-i$. Here $x \xi$ stands for the scalar product in \rn. Both $F$ and $F^{-1}$ are extended to $S'(\rn)$ in the
standard way. Let $\vp_0 \in S(\rn)$ with
\begin{\eq}   \label{2.3}
\vp_0 (x) =1 \ \text{if $|x|\le 1$} \quad \text{and} \quad \vp_0 (x) =0 \ \text{if $|x| \ge 3/2$},
\end{\eq}
and let
\begin{\eq}   \label{2.4}
\vp_k (x) = \vp_0 (2^{-k} x) - \vp_0 (2^{-k+1} x ), \qquad x\in \rn, \quad k\in \nat.
\end{\eq}
Since
\begin{\eq}   \label{2.5}
\sum^\infty_{j=0} \vp_j (x) =1 \qquad \text{for} \quad x\in \rn,
\end{\eq}
$\vp =\{ \vp_j \}^\infty_{j=0}$ forms a dyadic resolution of unity. The entire analytic functions $(\vp_j \wh{f} )^\vee (x)$ make sense pointwise in $\rn$ for any $f\in S'(\rn)$.

We recall the well-known definitions of the {\em inhomogeneous spaces} $\Bs (\rn)$ and $\Fs(\rn)$.
\\[0.1cm]
(i) {\em Let}
\begin{\eq}   \label{2.6}
0<p \le \infty, \qquad 0<q \le \infty, \qquad s \in \real.
\end{\eq}
{\em Then $\Bs (\rn)$ is the collection of all $f \in S' (\rn)$ such that}
\begin{\eq}   \label{2.7}
\| f \, | \Bs (\rn) \|_{\vp} = \Big( \sum^\infty_{j=0} 2^{jsq} \big\| (\vp_j \widehat{f})^\vee \, | L_p (\rn) \big\|^q \Big)^{1/q} < \infty
\end{\eq}
({\em with the usual modification if $q= \infty$}). 
\\[0.1cm]
(ii) {\em Let}
\begin{\eq}   \label{2.8}
0<p<\infty, \qquad 0<q \le \infty, \qquad s \in \real.
\end{\eq}
{\em Then $F^s_{p,q} (\rn)$ is the collection of all $f \in S' (\rn)$ such that}
\begin{\eq}   \label{2.9}
\| f \, | F^s_{p,q} (\rn) \|_{\vp} = \Big\| \Big( \sum^\infty_{j=0} 2^{jsq} \big| (\vp_j \wh{f})^\vee (\cdot) \big|^q \Big)^{1/q} \big| L_p (\rn) \Big\| < \infty
\end{\eq}
({\em with the usual modification if $q=\infty$}).
\\[0.1cm]
(iii) {\em Let $0<q\le \infty$, $s\in \real$. Then $F^s_{\infty,q} (\rn)$ is the collection of all $f\in S'(\rn)$ such that}
\begin{\eq}   \label{2.10}
\| f \, |\ F^s_{\infty,q} (\rn) \|_{\vp} = \sup_{J\in \no, M\in \zn} 2^{Jn/q} \Big(\int_{Q_{J,M}} \sum_{j \ge J} 2^{jsq} \big| (\vp_j \wh{f} )^\vee (x) \big|^q \, \di x \Big)^{1/q} <\infty
\end{\eq}
({\em with the usual modification if $q=\infty$}).

\begin{remark}   \label{R2.1}
As usual $\As (\rn)$ with $A \in \{B,F \}$ means $\Bs (\rn)$ and $\Fs (\rn)$. The theory of these {\em inhomogeneous spaces} including special cases and
their history may be found in \cite{T83, T92, T06}. The above definition of $F^s_{\infty,q} (\rn)$ goes back to \cite[Section
12, (12.8), p.\,133]{FrJ90}, where
\begin{\eq}   \label{2.11}
F^s_{\infty, \infty} (\rn) = B^s_{\infty, \infty} (\rn), \qquad s \in \real.
\end{\eq}
We do not deal in this paper with these inhomogeneous spaces. We only mention that they are independent of $\vp$ (equivalent quasi-norms). This
justifies to write $\As (\rn)$ instead of $\As (\rn)_{\vp}$. But it is crucial for our later considerations that the above Fourier-analytical 
definitions of $\As (\rn)$ can be replaces by definitions in terms of heat kernels. We give a description.
\end{remark}

Let $w\in S'(\rn)$. Then
\begin{\eq}   \label{2.12}
W_t w(x) = \frac{1}{(4 \pi t)^{n/2}} \int_{\rn} e^{- \frac{|x-y|^2}{4t}} w(y) \, \di y = \frac{1}{(4 \pi t)^{n/2}} \Big( w, e^{- \frac{|x-\cdot|^2}{4t}}
\Big), \quad t>0,
\end{\eq}
$x\in \rn$, is the well-known Gauss-Weierstrass semi-group which can be written on the Fourier side as
\begin{\eq}   \label{2.13}
\wh{W_t w}(\xi) = e^{-t |\xi|^2} \wh{w}(\xi), \qquad \xi \in \rn, \quad t>0.
\end{\eq}
The Fourier transform is taken with respect to the space variables $x\in \rn$. Of course, both \eqref{2.12}, \eqref{2.13} must be interpreted in the    context of $S'(\rn)$. But we recall that \eqref{2.12} makes sense pointwise: It is the convolution of $w\in S' (\rn)$ and $g_t (y) = (4\pi t)^{-n/2} e^{-\frac{|y|^2}{4t}} \in S(\rn)$. In particular,
\begin{\eq}   \label{2.14}
w \ast g_t \in C^\infty (\rn), \qquad |(w \ast g_t )(x)| \le c_t \big(1+|x|^2 \big)^{N/2}, \quad x\in \rn,
\end{\eq}
for some $c_t >0$ and, some $N \in \nat$. Further explanations and related references may be found in \cite[Section 4.1]{T14}. 

We give a description how the above Fourier-analytical definition of $\As (\rn)$ with $s<0$ can be replaced by corresponding characterizations in terms
of heat kernels. Let
\begin{\eq}   \label{2.15}
s<0 \quad \text{and} \quad 0<p,q\le \infty \qquad \text{(with $p<\infty$ for $F$-spaces)}.
\end{\eq}
Then
\begin{\eq}   \label{2.16}
\| f \, | \Bs (\rn) \| = \Big( \int^1_0 t^{-\frac{sq}{2}} \| W_t f \, | L_p (\rn) \|^q \, \frac{\di t}{t} \Big)^{1/q}
\end{\eq}
and
\begin{\eq}   \label{2.17}
\| f \, |F^s_{p,q} (\rn) \| \sim \Big\| \Big( \int^1_0 t^{-\frac{sq}{2}} \big| W_t f(\cdot)\big|^q \frac{\di t}{t} \Big)^{1/q} | L_p (\rn) \Big\|
\end{\eq}
(usual modification if $q=\infty$) are {\em admissible}  (characterizing) equivalent quasi-norms in the respective spaces.
More precisely: $f\in S'(\rn)$ belongs to $\Bs (\rn)$ if, and only if, the right-hand of \eqref{2.16} is finite. Similarly
for $F^s_{p,q} (\rn)$ with $p<\infty$. This is essentially covered by
\cite[Theorem, p.\,152]{T92}. Additional explanations may be found in \cite[p.\,14]{T15}. If $s<0$, $p=\infty$ and $0<q\le \infty$ then \eqref{2.17}
must be replaced by
\begin{\eq}  \label{2.18}
\| f \, | F^s_{\infty,q} (\rn) \| =  \sup_{x\in \rn, 0<t<1} \Big( t^{-n/2} \int^t_0 \int_{|x-y| \le \sqrt{t}} \tau^{-sq/2} |W_\tau f(y)|^q \, \di y \, \frac{\di \tau}{\tau} \Big)^{1/q}.
\end{\eq}
It is again an {\em admissible} equivalent quasi-norm: $f \in S'(\rn)$ belongs to $F^s_{\infty,q} (\rn)$ if, and only if, the right-hand side of 
\eqref{2.18} is finite. This is covered by \cite{Ryc99} which in turn is based on \cite{BPT96, BPT97}. Here \eqref{2.18} with $q=\infty$ means
\begin{\eq}   \label{2.19}
\| f \, | F^s_{\infty, \infty} (\rn) \| = \sup_{x \in \rn, 0<\tau <t<1} t^{-n/2} \tau^{-s/2} \, \int_{|x-y| \le \sqrt{t}} | W_\tau f(y) | \, \di y.
\end{\eq}
One has always
\begin{\eq}  \label{2.20}
B^s_{p,p} (\rn) = F^s_{p,p} (\rn), \qquad s<0, \quad 0<p\le \infty.
\end{\eq}
This well-known assertion follows obviously from \eqref{2.7} compared with \eqref{2.9} or \eqref{2.16} compared with \eqref{2.17} if $p<\infty$. The 
case $p=\infty$ is covered by \eqref{2.11}. But the resulting equivalence of \eqref{2.19} and \eqref{2.16} with $p= q= \infty$ is a special case of
the equivalence
\begin{\eq}  \label{2.21}
\begin{aligned}
\| f \, | B^s_{\infty, \infty} (\rn) \| &= \sup_{x\in \rn, 0<t<1} t^{-\frac{s}{2}} \big| W_t f(x) \big| \\
&\sim \sup_{x\in \rn, 0<\tau<t<1} \tau^{- \frac{s}{2}} \Big( t^{-\frac{n}{2}} \int_{|x-y| \le  \sqrt{t}} |W_\tau f(y) |^r \di y \Big)^{1/r}
\end{aligned}
\end{\eq}
for any $r$, $0<r<\infty$. The assertion itself goes back to \cite{BuiT00}. A short direct proof (in the framework of tempered homogeneous spaces) may
be found in \cite[p.\,47/48]{T15}.

Usually one introduces the inhomogeneous spaces $\As (\rn)$, $A \in \{ B,F \}$, in terms of Fourier-analytical decompositions  according to 
\eqref{2.6}--\eqref{2.10}. But thermic (caloric, in terms of Gauss-Weierstrass semi-groups) and also harmonic (in terms of Cauchy-Poisson semi-groups)
characterizations of these spaces have a long history beginning with \cite{Tai64, Tai65}. In \cite[p.\,13]{T15} we collected some further relevant
papers and books. This will not be repeated here. But we wish to mention \cite{BuiC16} which is the most recent contribution to this field of research
characterizing homogeneous spaces $\dot{A}^s_{p,q} (\rn)$ in terms of Cauchy-Poisson semi-groups discussing to which extent the ambiguity modulo
polynomials can be avoided (especially if $s<n/p$). 

The main reason for inserting the above material is the following. Let $s<0$. Then it does not matter of whether one introduces the spaces $\As (\rn)$
Fourier-analytically as in \eqref{2.6}--\eqref{2.10} or in terms of heat kernels according to \eqref{2.15}--\eqref{2.18}. All quasi-norms are 
{\em admissible} (or characterizing) in the following sense: Let $A(\rn)$ be a quasi-Banach spaces with
\begin{\eq}   \label{2.22}
A(\rn) \hra S'(\rn) \qquad \text{(continuous embedding)}.
\end{\eq}
An (equivalent) quasi-norm $\| \cdot \, | A(\rn) \|$ is called {\em admissible} if it makes sense to test any $f\in S'(\rn)$ for whether it belongs
to the corresponding space $A(\rn)$ or not (which means whether the related quasi-norm is finite or infinite). Within a given fixed quasi-Banach space
$A(\rn)$ further equivalent quasi-norms are called {\em domestic}. A simple but nevertheless illuminating example is
\begin{\eq}   \label{2.23}
L_p (\rn) = F^0_{p,2} (\rn), \qquad 1<p<\infty.
\end{\eq}
Then $\| \cdot \, | F^0_{p,2} (\rn) \|_{\vp}$ according to \eqref{2.9} is an admissible norm, whereas \eqref{2.1} is a domestic norm. We discussed this
point in greater detail in \cite[Section 1.3, pp.\,5/6]{T15}. 

If one switches from inhomogeneous spaces to homogeneous spaces then there is a decisive difference between the Fourier-analytical homogeneous 
counterparts of \eqref{2.6}--\eqref{2.10} on the one hand and the possibility to introduce related homogeneous spaces in the framework of $S'(\rn)$
in terms of admissible quasi-norms based on heat kernels. Preference is given to the homogeneous counterparts of \eqref{2.16}--\eqref{2.18} first with
$s,p,q$ as in \eqref{2.15} extended afterwards to the distinguished strip \eqref{1.2}, incorporating some limiting cases. The Fourier-analytical versions
are domestic quasi-norms in the related spaces. This is sufficient to work with them in applications to Navier-Stokes or Keller-Segel equations now on
the safe topological ground of $S'(\rn)$. This was the main aim of \cite{T15} and will be continued now. We take over those definitions and basic
assertions from \cite{T15} needed later on. In this sense we try to make this paper independently readable.

\subsection{Spaces with negative smoothness}   \label{S2.2}
First we complement \cite[Section 3.1]{T15} where we introduced the tempered homogeneous spaces $\Aas (\rn)$ with $s<0$ in the framework of the dual
pairing $\big(S (\rn), S' (\rn) \big)$ based on \eqref{2.12}--\eqref{2.14}. For sake of completeness  we repeat some definitions and basic assertions.

\begin{definition}   \label{D2.2}
{\upshape (i)} Let $s<0$ and $0<p,q \le \infty$. Then $\Bas (\rn)$ collects all $f\in S' (\rn)$ such that
\begin{\eq}   \label{2.24}
\| f \, | \Bas (\rn) \| = \Big( \int^\infty_0 t^{-sq/2} \| W_t f \, | L_p (\rn) \|^q \, \frac{\di t}{t} \Big)^{1/q}
\end{\eq}
is finite $($usual modification if $q= \infty)$.
\\[0.1cm]
{\upshape (ii)} Let $s<0$, $0<p<\infty$, $0<q\le \infty$. Then $\Fas (\rn)$ collects all $f\in S' (\rn)$ such that
\begin{\eq}   \label{2.25}
\| f \, | \Fas (\rn) \| = \Big\| \Big( \int^\infty_0 t^{-sq/2} \big| W_t f(\cdot) \big|^q \, \frac{\di t}{t} \Big)^{1/q} \, | L_p (\rn) \Big\|
\end{\eq}
is finite $($usual modification if $q=\infty)$.
\\[0.1cm]
{\upshape (iii)} Let $s<0$ and $0<q \le \infty$. Then $\Fa{}^s_{\infty,q} (\rn)$ collects all $f\in S' (\rn)$ such that
\begin{\eq}   \label{2.26}
\| f \, | \Fa{}^s_{\infty,q} (\rn) \| = \sup_{x\in \rn, t>0} \Big( t^{-n/2} \int^t_0 \int_{|x-y| \le \sqrt{t}} \tau^{-sq/2} \big| W_\tau f(y) \big|^q
\, \di y \, \frac{\di \tau}{\tau} \Big)^{1/q}
\end{\eq}
is finite $($modification if $q=\infty$ as explained below$)$.
\end{definition}

\begin{remark}   \label{R2.3}
This coincides with \cite[Definition 3.1, p.\,46]{T15}. There one finds also a few further discussions and explanations. As usual $\Aas (\rn)$ with
$A \in \{B,F \}$ means $\Bas (\rn)$ and $\Fas (\rn)$. Quite obviously, \eqref{2.24}--\eqref{2.26} is the homogeneous counterpart of the {\em admissible}
(characterizing) quasi-norms \eqref{2.16}--\eqref{2.19} of the related {\em inhomogeneous} spaces $\As (\rn)$, including the interpretation of 
\eqref{2.26} with $q=\infty$,
\begin{\eq}   \label{2.27}
\| f \, | \Fa{}^s_{\infty, \infty} (\rn) \| = \sup_{x\in \rn, 0<\tau <t} t^{-n/2} \tau^{-s/2} \int_{|x-y| \le \sqrt{t}} | W_\tau f(y) | \, \di y.
\end{\eq}
Furthermore there are counterparts of \eqref{2.20} and \eqref{2.21}, that is
\begin{\eq}  \label{2.28}
\Ba{}^s_{p,p} (\rn) = \Fa{}^s_{p,p} (\rn), \qquad s<0, \quad 0<p\le \infty
\end{\eq}
and
\begin{\eq}  \label{2.29}
\begin{aligned}
\| f \, | \Ba{}^s_{\infty, \infty} (\rn) \| &= \sup_{x\in \rn, t>0} t^{-\frac{s}{2}} \big| W_t f(x) \big| \\
&\sim \sup_{x\in \rn, 0<\tau<t} \tau^{- \frac{s}{2}} \Big( t^{-\frac{n}{2}} \int_{|x-y| \le  \sqrt{t}} |W_\tau f(y) |^r \di y \Big)^{1/r}
\end{aligned}
\end{\eq}
for any $r$ with $0<r<\infty$. A short direct proof may be found in \cite[p.\,47]{T15}.
\end{remark}

We formulate some further properties of the spaces $\Aas (\rn)$ following again \cite{T15}.  Let $A(\rn)$ be a quasi-normed
space in $S' (\rn)$ with $A(\rn) \hra S'(\rn)$. Then $A(\rn)$ is said to have the {\em Fatou property} if there is a positive constant $c$ such that
from
\begin{\eq}   \label{2.30}
\sup_{j \in \nat} \| g_j \, | A(\rn) \| < \infty \qquad \text{and} \qquad g_j \to g \quad \text{in $S' (\rn)$}
\end{\eq}
it follows $g\in A(\rn)$ and
\begin{\eq}   \label{2.31}
\| g \, | A(\rn) \| \le c \, \sup_{j \in \nat} \| g_j \, | A(\rn) \|.
\end{\eq}
Further explanations may be found in \cite[p.\,48]{T15}. We need the homogeneous counterpart of the (inhomogeneous) dyadic resolution of unity according
\eqref{2.3}--\eqref{2.5}. Let $\vp_0$ be as in \eqref{2.3} and
\begin{\eq}   \label{2.32}
\vp^j (x) = \vp_0 (2^{-j} x) - \vp_0 (2^{-j+1} x), \qquad x\in \rn, \quad j \in \ganz.
\end{\eq}
Then
\begin{\eq}   \label{2.33}
\sum_{j\in \ganz} \vp^j (x) =1 \qquad \text{for} \quad x \in \rn \setminus \{0 \}.
\end{\eq}
Recall that $\hra$ means continuous embedding. 

\begin{theorem}   \label{T2.4}
{\upshape (i)}
The spaces $\Aas (\rn)$ with $s<0$ and $0<p,q \le \infty$ according to Definition \ref{D2.2} are quasi-Banach spaces $($Banach spaces if $p\ge 1$, $q
\ge 1)$. They have the Fatou property. Furthermore,
\begin{\eq}   \label{2.34}
\Aas (\rn) \hra \As (\rn) \hra S' (\rn)
\end{\eq}
and
\begin{\eq}   \label{2.35}
\| f (\lambda \cdot) \, | \Aas (\rn) \| = \lambda^{s- \frac{n}{p}} \| f \, | \Aas (\rn) \|, \qquad \lambda >0.
\end{\eq}
{\upshape (ii)}
Let
\begin{\eq}   \label{2.36}
s<0 \quad \text{and} \quad 0<p,q\le \infty \qquad \text{$($with $p<\infty$ for $F$-spaces$)$}.
\end{\eq}
Then
\begin{\eq}   \label{2.37}
\Big(\sum^\infty_{j=-\infty} 2^{jsq} \| (\vp^j \wh{f} )^\vee | L_p (\rn) \|^q \Big)^{1/q}, \qquad f\in \Bas (\rn),
\end{\eq}
are equivalent {\em domestic} quasi-norms in $\Bas (\rn)$ and
\begin{\eq}   \label{2.38}
\Big\| \Big( \sum^\infty_{j=-\infty} 2^{jsq} \big| (\vp^j \wh{f} )^\vee (\cdot) \big|^q \Big)^{1/q} | L_p (\rn) \Big\|, \qquad f \in \Fas (\rn),
\end{\eq}
are equivalent {\em domestic} quasi-norms in $\Fas (\rn)$.
\end{theorem}

\begin{remark}   \label{R2.5}
This coincides with \cite[Theorem 3.3, p.\,48/49]{T15}. There one finds also a detailed proof. For the spaces $\Fa{}^s_{\infty,q} (\rn)$ one has a
counterpart of \eqref{2.10} with $J\in \ganz$ in place of $J\in \no$ (equivalent {\em domestic} quasi-norms). This is covered by \cite[(2.32), (2.54),
pp.\,12, 17]{T15}. We explained at the end of Section \ref{S2.1} what is meant by (equivalent) {\em domestic} quasi-norms in a given space. But this
can also be applied to a fixed family (or community) of spaces $\Aas (\rn)$ and also to (linear and nonlinear) mappings between members of this family.
At the beginning and at the end one relies on admissible (characterizing, defining) quasi-norms working in between with suitable domestic quasi-norms.
There is no need to bother any longer about ambiguities modulo polynomials as in the framework of the dual pairing $\big( \dot{S}(\rn), \dot{S}' (\rn)
\big)$. We refer again for greater details to \cite[Section 1.3]{T15}.
\end{remark}

We are mainly interested in quasi-Banach spaces $A(\rn)$ in the framework of the dual pairing $\big( S(\rn), S' (\rn) \big)$ such that
\begin{\eq}  \label{2.39}
S(\rn) \hra A(\rn) \hra S' (\rn).
\end{\eq}
Recall that $\hra$ means linear continuous embedding. So far we have \eqref{2.34}. The motivation to ask for $S(\rn) \hra A(\rn)$ comes again from 
applications. Admitted initial data for Navier-Stokes equations \eqref{1.4}, \eqref{1.5} or for Keller-Segel equations \eqref{1.7}--\eqref{1.10} are
nowadays  quite often assumed to be elements of (homogeneous) spaces $A(\rn)$ in the framework of $\big( S(\rn), S'(\rn) \big).$ Then the left-hand
side of \eqref{2.39} ensures that functions belonging to $D(\rn)$ or $S(\rn)$ are admitted initial data what seems to be a minimal request. A first
step in this direction for the spaces $\Aas (\rn)$ as introduced in Definition \ref{D2.2} has been done in \cite[Theorem 3.5, p.\,52]{T15}. This will
now be complemented having in particular a closer look at limiting spaces (again motivated by applications as described at the end of Section \ref{S1}
in connection with Keller-Segel equations especially in the plane $\real^2$).

Let $\Aas (\rn)$ with $A \in \{B,F \}$ be the spaces introduced in Definition \ref{D2.2}. Recall $D(\rn) = C^\infty_0 (\rn)$ (compactly supported
$C^\infty$-functions in $\rn$).

\begin{theorem}   \label{T2.6}
Let $n\in \nat$, $0<p\le \infty$, $0<q\le \infty$ and $s<0$. Then
\begin{\eq}  \label{2.40}
S(\rn) \hra \Bas (\rn)
\end{\eq}
if, and only if,
\begin{\eq}  \label{2.41}
1<p \le \infty, \qquad
\begin{cases}
s> n \big( \frac{1}{p} - 1 \big), &0<q \le \infty, \\
s= n \big(\frac{1}{p} - 1 \big), &q=\infty,
\end{cases}
\end{\eq}
and
\begin{\eq}  \label{2.42}
S(\rn) \hra \Fas (\rn)
\end{\eq}
if, and only if,
\begin{\eq}   \label{2.43}
1<p \le \infty, \qquad s>n \big( \frac{1}{p} - 1 \big), \qquad 0<q \le \infty.
\end{\eq}
This assertion remains valid if one replaces $S(\rn)$ by $D(\rn)$.
\end{theorem}

\begin{proof}
{\em Step 1.} According to \cite[Theorem 3.5, p.\,52]{T15} one has \eqref{2.40} and \eqref{2.42} if $1<p\le \infty$, $0<q\le \infty$ and $s > n \big(
\frac{1}{p} - 1 \big)$. It remains to prove
\begin{\eq}   \label{2.44}
S(\rn) \hra \Ba{}^{n ( \frac{1}{p} - 1 )}_{p, \infty} (\rn), \qquad 1<p \le \infty,
\end{\eq}
and that there is no embedding of $S(\rn)$ into $\Aas (\rn)$ in all other cases.
\\[0.1cm]
{\em Step 2.} We prove \eqref{2.44}. Let $f\in L_1 (\rn)$. Then one has by \eqref{2.12}
\begin{\eq}  \label{2.45}
(4 \pi t)^{\frac{n}{2}} \| W_t f \, | L_p (\rn) \| \le \| e^{-\frac{|\cdot|^2}{4t}} \, | L_p (\rn) \| \cdot \| f \, | L_1 (\rn) \| = c_p \, 
t^{\frac{n}{2p}} \, \| f \, | L_1 (\rn) \|
\end{\eq}
and by \eqref{2.24}
\begin{\eq}   \label{2.46}
S(\rn) \hra L_1 (\rn) \hra \Ba{}^{n (\frac{1}{p} - 1 )}_{p,\infty} (\rn).
\end{\eq}
This proves \eqref{2.44}.
\\[0.1cm]
{\em Step 3.} We prove the negative assertions of the above theorem assuming first $s = n \big( \frac{1}{p} - 1 \big)$, $1<p \le \infty$. Let $f \in
S(\rn)$, $f(x) \ge 0$,
\begin{\eq}   \label{2.47}
f(x) =1 \ \text{if} \ |x| \le 1/2 \quad \text{and} \quad f(x) =0 \ \text{if} \ |x| \ge 1.
\end{\eq}
By \eqref{2.12} one has for some $c>0$
\begin{\eq}   \label{2.48}
t^{n/2} W_t f(x) \ge c\, e^{- \frac{|x|^2}{16 t}}, \qquad |x| \ge 2, \quad 0<t<\infty.
\end{\eq}
Let $0<q<\infty$. Then
\begin{\eq}  \label{2.49}
\| f \, | \Ba{}^{-n}_{\infty, q} (\rn) \|^q = \int^\infty_0 t^{\frac{nq}{2}} \| W_t f \, | L_\infty (\rn) \|^q \frac{\di t}{t} \ge c \, \int^\infty_1
\frac{\di t}{t} = \infty.
\end{\eq}
In particular $f\not\in \Ba{}^{-n}_{\infty, q} (\rn)$. The domestic quasi-norms \eqref{2.37}, \eqref{2.38} ensure that suitable well-known embeddings
between the inhomogeneous spaces $\As (\rn)$ can be transferred to the tempered homogeneous spaces $\Aas (\rn)$ by the same arguments as there. In
particular
\begin{\eq}  \label{2.50}
\Ba{}^{n (\frac{1}{p} - 1)}_{p,q} (\rn) \hra \Ba{}^{-n}_{\infty,q} (\rn), \qquad 1<p<\infty,
\end{\eq}
shows that $f\not\in \Ba{}^{n(\frac{1}{p} - 1)}_{p,q} (\rn)$ if $0<q<\infty$. This assertion can be extended to those $F$-spaces which are covered by
\begin{\eq}   \label{2.51}
\Fa{}^{n(\frac{1}{p} - 1)}_{p,q} (\rn) \hra \Ba{}^{n(\frac{1}{p} -1 )}_{p, \max(p,q)} (\rn), \qquad p<\infty, \quad q<\infty.
\end{\eq}
If one uses the more sophisticated embedding
\begin{\eq}  \label{2.52}
\Fa{}^{n(\frac{1}{p} -1)}_{p,q} (\rn) \hra \Ba{}^{-n}_{\infty,p} (\rn), \qquad 1<p<\infty,
\end{\eq}
then it follows that also
\begin{\eq}   \label{2.53}
f \not\in \Fa{}^{n (\frac{1}{p} - 1)}_{p, \infty} (\rn), \qquad 1<p<\infty.
\end{\eq}
We return to embeddings of type \eqref{2.52} in Remark \ref{R2.7} below and give now a direct proof of \eqref{2.53} based on \eqref{2.25} with $s = n
(\frac{1}{p} - 1 )$. Using \eqref{2.48} with $t \sim |x|^2$ one has
\begin{\eq}   \label{2.54}
\begin{aligned}
\| f \, | \Fa{}^s_{p,\infty} (\rn) \|^p &= \big\| \sup_{t>0} t^{-s/2} W_t f(\cdot) \, |L_p (\rn) \big\|^p \\
&\ge c \, \int_{|x| \ge 2} |x|^{-sp -np} \, \di x = c \, \int_{|x|  \ge 2} |x|^{-n} \, \di x =\infty.
\end{aligned}
\end{\eq}
This proves \eqref{2.53} independently of \eqref{2.52}. It remains to show that
\begin{\eq}   \label{2.55}
f \not\in \Fa{}^{-n}_{\infty, q} (\rn), \qquad 0<q<\infty.
\end{\eq}
Let $4 \le 4 \sqrt{t} \le |x| \le 8 \sqrt{t}$ and $|x-y| \le \sqrt{t}$. Then $2 \sqrt{t} \le |y| \le 9 \sqrt{t}$. By \eqref{2.48} and $t/2 \le \tau \le
t$ one has $|W_\tau f(y)| \ge c t^{-n/2} \sim \tau^{-n/2}$ and
\begin{\eq}   \label{2.56}
t^{-n/2} \int^t_{t/2} \int_{|x-y| \le \sqrt{t}} \tau^{nq/2} |W_\tau f(y) |^q \, \di y \frac{\di \tau}{\tau} \ge c' \, \log t.
\end{\eq}
Then \eqref{2.55} follows from \eqref{2.26} and \eqref{2.56}. This covers all cases of the theorem with $s= n (\frac{1}{p} - 1)$ where one can replace
$S(\rn)$ by $D(\rn)$.
\\[0.1cm]
{\em Step 4.} Let $s<-n$, $0<q\le \infty$ and $f$ as above. Then one has by \eqref{2.24} and \eqref{2.48}
\begin{\eq}   \label{2.57}
f \not\in \Ba{}^s_{\infty, q} (\rn), \qquad 0<q\le \infty.
\end{\eq}
Similarly one can extend \eqref{2.55} based on \eqref{2.56} to $\Fa{}^s_{\infty,q} (\rn)$ with $s<-n$ and $0<q<\infty$. Let
\begin{\eq}   \label{2.58}
0<p<\infty, \quad 0<q\le \infty \quad \text{and} \quad s< \min \Big( 0, n \big( \frac{1}{p} - 1 \big) \Big).
\end{\eq}
Then
\begin{\eq}   \label{2.59}
\Fas (\rn) \hra \Ba{}^s_{p, \max(p,q)} (\rn) \hra \Ba{}^{s - \frac{n}{p}}_{\infty, \infty} (\rn).
\end{\eq}
Using \eqref{2.57} with $s - \frac{n}{p} < -n$ in place of $s$ it follows that the above function $f$ does not belong to any of these spaces. This
completes the proof of the theorem.
\end{proof}

\begin{remark}  \label{R2.7}
The spaces $\Aas(\rn)$ are introduced in Definition \ref{D2.2} in terms of admissible quasi-norms based on heat kernels. The 
{\em domestic} quasi-norms in \eqref{2.37} and \eqref{2.38} give the possibility to carry over many well-known properties for the usual homogeneous
spaces $\dot{A}^s_{p,q} (\rn)$ in the framework of $\big( \dot{S} (\rn), \dot{S}' (\rn) \big)$ where $\dot{S}' (\rn)$ is the dual of $\dot{S} (\rn)$
according to \eqref{1.1a}. This applies to some embeddings, real and complex interpolation and further properties. We refer the reader to
\cite[Chapter 5]{T83} and \cite{T15}. The embeddings \eqref{2.50} and \eqref{2.59} may be considered as examples. The more sophisticated embedding
\eqref{2.52} is covered by the following so-called Franke-Jawerth property: Let $0< p_0 <p< p_1 \le \infty$, $s_0 <0$ (so far) and
\begin{\eq}  \label{2.60}
s_0 - \frac{n}{p_0} = s - \frac{n}{p} = s_1 - \frac{n}{p_1}, \qquad 0<q\le \infty, \quad 0<u \le \infty, \quad 0<v \le \infty.
\end{\eq}
Then
\begin{\eq}   \label{2.61}
\Ba{}^{s_0}_{p_0, u} (\rn) \hra \Fas (\rn) \hra \Ba{}^{s_1}_{p_1,v} (\rn)
\end{\eq}
if, and only if, $0<u \le p \le v \le \infty$. This follows from the corresponding assertions for the related inhomogeneous spaces, going back in its
final version to \cite{SiT95} based on \cite{Jaw77, Fra86}. We return to this assertion in Theorem \ref{T3.1} below. We indicated in \eqref{2.52}
that one can use these embeddings for the above  purposes (but we gave a direct proof of \eqref{2.53}).
\end{remark}

\subsection{Spaces in the distinguished strip}   \label{S2.3}
As indicated in the Introduction we are mainly interested in tempered homogeneous spaces in the distinguished strip \eqref{1.2} having now a closer
look what happens at $s = n(\frac{1}{p} - 1)$ and $s= \frac{n}{p}$. We wish to ensure \eqref{2.39}. If $s<0$ then we obtained in the Theorems 
\ref{T2.4}, \ref{T2.6} rather final satisfactory answers. To extend this theory to $s \ge 0$ we modify the corresponding approach in \cite{T15}
appropriately, now incorporating some limiting cases. Let 
\begin{\eq}   \label{2.62}
\Cas (\rn) = \Ba{}^s_{\infty, \infty} (\rn), \qquad s<0.
\end{\eq}
By Theorem \ref{T2.6} one has
\begin{\eq}   \label{2.63}
S(\rn) \hra \Cas (\rn) \quad \text{if, and only if,} \quad -n \le s <0.
\end{\eq}
Let again $W_t w$ be the Gauss-Weierstrass semi-group according to \eqref{2.12}, \eqref{2.13}, where $t \ge 0$ and $w\in S'(\rn)$. Let, as nowadays
usual, $\pa_t = \pa/\pa t$ and $\pa^m_t$, $m \in \no$ be the corresponding iterations, $\pa^0_t w = w$. Recall that
\begin{\eq}   \label{2.64}
\| f \, | \Cas (\rn) \| = \sup_{x\in \rn, t>0} t^{-s/2} \, | W_t f(x) |, \qquad s<0,
\end{\eq}
is an {\em admissible} (defining, characterizing) norm which means that any $f\in S'(\rn)$ can be tested of whether it belongs to this space or not. Let
$\vp^j$ be as in \eqref{2.32}. Then
\begin{\eq}   \label{2.65}
\sup_{j\in \ganz, x\in \rn} 2^{js} \big| \big( \vp^j \wh{f}\, \big)^\vee (x) \big|, \qquad s<0,
\end{\eq}
and
\begin{\eq} \label{2.66}
\sup_{x\in \rn, t>0} t^{m - \frac{s}{2}} \big| \pa^m_t W_t f(x) \big|, \qquad m \in \no, \quad s<0,
\end{\eq}
are {\em domestic} norms on $\Cas (\rn)$, which means that they are equivalent norms on the given space $\Cas (\rn)$. Here \eqref{2.65} is covered by
Theorem \ref{T2.4} and \eqref{2.66} by \cite[Proposition 2.12, p.\,20]{T15}. 

To provide a better understanding we recall characterizations of the inhomogeneous spaces $\As (\rn)$ introduced in \eqref{2.3}--\eqref{2.9} in terms of
heat kernels. Let $0<p \le \infty$ ($p<\infty$ for $F$-spaces), $0<q\le \infty$, $s\in \real$ and $s/2 <m \in \no$. Then $\Bs (\rn)$ collects all $f \in
S'(\rn)$ such that
\begin{\eq}   \label{2.66a}
\| f \, | \Bs (\rn) \|_m =  \| W_1 f \, | L_p (\rn) \| +
\Big( \int^1_0 t^{(m- \frac{s}{2})q} \big\| \pa^m_t W_t f \, | L_p (\rn) \big\|^q \, \frac{\di t}{t} \Big)^{1/q}
\end{\eq}
is finite and $\Fs (\rn)$ collects all $f\in S' (\rn)$ such that
\begin{\eq}   \label{2.66b}
\| f \, |\Fs (\rn) \|_m =  \| W_1 f \, | L_p (\rn)\| +
\Big\| \Big( \int^1_0 t^{(m- \frac{s}{2})q} \big| \pa^m_t W_t f(\cdot) \big|^q \, \frac{\di t}{t} \Big)^{1/q} | L_p (\rn) \Big\|
\end{\eq}
is finite (usual modification if $q=\infty$). This is covered by \cite[Theorem 2.6.4, p.\,152]{T92}. According to \cite[Remark 2.6.4, p.\,155]{T92}
one can replace $\int^1_0$ in \eqref{2.66a}, \eqref{2.66b} by $\int^\infty_0$ ($\sup_{0<t<1}$ by $\sup_{t>0}$) if $s> \sigma_p = \max \big( 0, n 
(\frac{1}{p} - 1 ) \big)$. 

\begin{definition}   \label{D2.8}
Let $n\in \nat$.
\\[0.1cm]
{\upshape (i)}
Let $0<p \le \infty$, $(p<\infty$ for $F$-spaces$),\, 0<q\le \infty$ and
\begin{\eq}   \label{2.67}
n \big( \frac{1}{p} - 1 \big) <s< \frac{n}{p},   \quad \qquad -\frac{n}{r} = s - \frac{n}{p}.
\end{\eq}
Let $s/2 <m \in \no$. Then $\Fas (\rn)$ collects all $f\in S' (\rn)$ such that
\begin{\eq}   \label{2.68}
\| f \, |\Fas (\rn) \|_m = \Big\| \Big( \int^\infty_0 t^{(m- \frac{s}{2})q} \big| \pa^m_t W_t f(\cdot) \big|^q \, \frac{\di t}{t} \Big)^{1/q} |
L_p (\rn) \Big\| + \|f \, | \Ca{}^{-n/r} (\rn) \|
\end{\eq}
is finite and $\Bas (\rn)$ collects all $f\in S'(\rn)$ such that
\begin{\eq}   \label{2.69}
\| f \, | \Bas (\rn) \|_m = \Big( \int^\infty_0 t^{(m- \frac{s}{2})q} \big\| \pa^m_t W_t f \, | L_p (\rn) \big\|^q \, \frac{\di t}{t} \Big)^{1/q}
+ \| f \, | \Ca{}^{-n/r} (\rn) \|
\end{\eq}
is finite $($usual modification if $q=\infty)$.
\\[0.1cm]
{\upshape (ii)} Let $0<p\le \infty$ and $s=n (\frac{1}{p} - 1 )$. Let $s/2 <m \in \no$. Then $\Ba{}^s_{p,\infty} (\rn)$ collects all $f\in S' (\rn)$ such that
\begin{\eq}   \label{2.70}
\| f \, | \Ba{}^s_{p,\infty} (\rn) \|_m = \sup_{t >0} t^{m - \frac{s}{2}} \, \| \pa^m_t W_t f \, | L_p (\rn) \| + \| f\, | \Ca{}^{-n} (\rn) \|
\end{\eq}
is finite.
\\[0.1cm] 
{\upshape (iii)} Let $0<p \le 1$, $0<q\le \infty$, $s= n/p$ and $s/2 <m \in \nat$. Then $\Fas (\rn)$ collects all $f\in S'  (\rn)$ such that
\begin{\eq}   \label{2.71}
\| f \, |\Fas (\rn) \|_m = \Big\| \Big( \int^\infty_0 t^{(m- \frac{s}{2})q} \big| \pa^m_t W_t f(\cdot) \big|^q \, \frac{\di t}{t} \Big)^{1/q} |
L_p (\rn) \Big\| + \sup |(f,\vp)|
\end{\eq}
is finite $($usual modification if $q=\infty)$, where the supremum is taken over all $\vp \in S(\rn)$ with $\| \vp \, | L_1 (\rn) \| \le 1$.
\\[0.1cm] 
{\upshape (iv)} Let $0<p<\infty$, $0<q\le 1$, $s=n/p$ and $s/2 <m \in \nat$. Then $\Bas (\rn)$ collects all $f\in S'(\rn)$ such that
\begin{\eq}   \label{2.72}
\| f \, | \Bas (\rn) \|_m = \Big( \int^\infty_0 t^{(m- \frac{s}{2})q} \big\| \pa^m_t W_t f \, | L_p (\rn) \big\|^q \, \frac{\di t}{t} \Big)^{1/q}
+ \sup |(f, \vp) |
\end{\eq}
is finite where the supremum is taken over all $\vp \in S(\rn)$ with $\| \vp \, | L_1 (\rn) \| \le 1$.
\end{definition}

\begin{remark}   \label{R2.9}
Part (i) coincides essentially with \cite[Definition 3.22, p.\,78]{T15}. We only remark that $1<r<\infty$ in \eqref{2.67}. Then $\Ca{}^{-n/r} (\rn) =
\Ba{}^{-n/r}_{\infty, \infty} (\rn)$ is covered by Theorem \ref{T2.6}. This applies also to $\Ca{}^{-n} (\rn)$ in \eqref{2.70}. We discuss the limiting
cases in (ii)--(iv). For this purpose we recall some properties of their inhomogeneous counterparts. Let 
\begin{\eq}   \label{2.73}
\sigma_p = n \big( \frac{1}{p} - 1 \big)_+, \qquad 0<p \le \infty,
\end{\eq}
where $a_+ = \max(0,a)$ if $a\in \real$. Let $\vp_0$ be as in \eqref{2.3} and $\vp^j$ be as in \eqref{2.32}. If $0<p,q \le \infty$ and $s> \sigma_p$
then one has by \cite[Theorem 2.3.3, p.\,98]{T92}
\begin{\eq}   \label{2.74}
\| f \, | \Bs (\rn) \| \sim \| (\vp_0 \wh{f} \, )^\vee | L_p (\rn) \| + \Big( \sum^\infty_{j= -\infty} 2^{jsq} \big\| (\vp^j \wh{f}\, )^\vee \, |
L_p (\rn) \big\|^q \Big)^{1/q}.
\end{\eq}
The incorporation of the terms with $j<0$, compared with \eqref{2.7}, is based on
\begin{\eq}   \label{2.75}
\| (\vp^j \wh{f}\,)^\vee \, | L_p (\rn) \| \le c \, 2^{-j \sigma_p} \, \| (\vp_0 \wh{f}\,)^\vee \, | L_p (\rn) \|
\end{\eq}
covered by \cite[(6), p.\,98]{T92} with a reference to \cite{T83}. But this argument remains valid if $s= \sigma_p$ and $q=\infty$. Then one can extend
\cite[Theorem 2.6.4, Remark 2.6.4, pp.\,152. 155]{T92} to these limiting cases. In particular if $0<p \le \infty$, $\sigma_p/2 <m \in \nat$ then
\begin{\eq}  \label{2.76}
\begin{aligned}
\| f \, | B^{\sigma_p}_{p,\infty} (\rn) \| & \sim \| (\vp_0 \wh{f} \, )^\vee \, | L_p (\rn)\| + \sup_{j\in \ganz} 2^{j \sigma_p} \| (\vp^j 
\wh{f}\,)^\vee \, | L_p (\rn) \| \\
&\sim \| W_1 f \, | L_p (\rn) \| + \sup_{t>0} t^{m - \frac{\sigma_p}{2}} \, \| \pa^m_t W_t f \, | L_p (\rn) \|.
\end{aligned}
\end{\eq}
As for the parts (iii) and (iv) we discuss the last terms in \eqref{2.71}, \eqref{2.72}. By H\"{o}lder's inequality  one has $g \in L_1 (\rn)$ if
$g(\cdot) \big(1+ |\cdot|^2 \big)^{\alpha/2} \in L_2 (\rn) = F^0_{2,2} (\rn)$, $\alpha >n/2$ (continuous embedding of the related spaces). Duality
within $\big( S(\rn), S'(\rn) \big)$ shows that $f$ in (iii), (iv) belongs at least to a weighted $L_2$-space. In particular $f$ is a regular tempered
distribution. The dual of $L_1 (\rn)$ is $L_\infty (\rn)$ and $S(\rn)$ is dense in $L_1 (\rn)$. This shows that the last terms in \eqref{2.71}, 
\eqref{2.72} equal $\| f \, | L_\infty (\rn) \|$. After this replacement one has related domestic quasi-norms. After this replacement one has related
domestic quasi-norms.
Next we comment on the inhomogeneous counterpart of the limiting spaces with $s= n/p$ in the parts (iii) and (iv) of the above definition. Let
\begin{\eq}   \label{2.77}
0<p<\infty, \qquad 0<q\le \infty, \qquad s=n/p.
\end{\eq}
Then
\begin{\eq}   \label{2.78}
B^{n/p}_{p,q} (\rn) \hra L_\infty (\rn) \qquad \text{if, and only if,} \quad 0<p<\infty, \quad 0<q\le 1,
\end{\eq}
and
\begin{\eq}   \label{2.79}
F^{n/p}_{p,q} (\rn) \hra L_\infty (\rn) \qquad \text{if, and only if,} \quad 0<p \le 1, \quad 0<q\le \infty.
\end{\eq}
One can replace $L_\infty (\rn)$ in \eqref{2.78}, \eqref{2.79} by $C(\rn)$, the space of all continuous bounded functions on \rn. The final version
goes back to \cite{SiT95} and may also be found in \cite[Theorem 11.4, p.\,170]{T01}. The restrictions for $p,q$ in the parts (iii) and (iv) originate
from \eqref{2.78}, \eqref{2.79}.
\end{remark}

After these preparations we can now complement \cite[Theorem 3.24, pp.\,79-81]{T15} where we collected basic properties of the spaces covered by part
(i) of Definition \ref{D2.8}. We explained in \eqref{2.30}, \eqref{2.31} what is meant by the {\em Fatou property}. Furthermore let $\vp^j$ in 
\eqref{2.32}, \eqref{2.33} be again the homogeneous dyadic resolution of unity in $\rn \setminus \{0 \}$. Recall that a quasi-norm is called
{\em admissible} if any $f\in S'(\rn)$ can be tested of whether it belongs to the respective space or not (what this means is explained in greater 
detail between \eqref{2.17} and \eqref{2.18}). Equivalent quasi-norms in a fixed quasi-Banach space are called {\em domestic}. As before $\As (\rn)$
with $A\in \{B,F \}$ stands for $\Bs (\rn)$ and $\Fs (\rn)$. Similarly $\Aas (\rn)$. In a given formula either all $A$ are $B$ or all $A$ are $F$. Let
$\sigma_p$ be as in \eqref{2.73} and $r$ be as in \eqref{2.67}.

\begin{theorem}   \label{T2.10}
Let $n \in \nat$.
\\[0.1cm]
{\upshape (i)} The spaces $\Aas (\rn)$,
\begin{\eq}  \label{2.80}
A=F \quad \text{with} \quad
\begin{cases}
0<p<\infty, \ 0<q\le \infty, &n (\frac{1}{p} - 1) <s< \frac{n}{p}, \\
0<p\le 1, \ \ \; 0<q\le \infty, &s = \frac{n}{p},
\end{cases}
\end{\eq}
and
\begin{\eq}  \label{2.81}
A=B \quad \text{with} \quad
\begin{cases}
0<p \le \infty, \ 0<q\le \infty, &n (\frac{1}{p} - 1 ) <s< \frac{n}{p}, \\
0<p \le \infty, \ q=\infty, &s = n (\frac{1}{p} - 1 ), \\
0<p<\infty, \ 0<q\le 1, & s= \frac{n}{p},
\end{cases}
\end{\eq}
as introduced in Definition \ref{D2.8} are quasi-Banach spaces $($Banach spaces if $p\ge 1$, $q\ge 1)$. Let $s/2 <m\in \no$. Then $\| f \, |\Fas (\rn)
\|_m$ according to \eqref{2.68}, \eqref{2.71} are equivalent {\em admissible} quasi-norms in $\Fas (\rn)$ and $\| f \, | \Bas (\rn) \|_m$ according
to \eqref{2.69}, \eqref{2.70}, \eqref{2.72} are equivalent {\em admissible} quasi-norms in $\Bas (\rn)$. All spaces have the Fatou property. 
Furthermore,
\begin{\eq}   \label{2.82}
S(\rn) \hra \Aas(\rn) \hra S'(\rn)
\end{\eq}
and
\begin{\eq}   \label{2.83}
\| f(\lambda \cdot)| \Aas (\rn) \| = \lambda^{s- \frac{n}{p}} \| f \, | \Aas (\rn) \|, \qquad \lambda >0.
\end{\eq}
In addition,
\begin{\eq}   \label{2.84}
\Big\| \Big( \int^\infty_0 t^{(m- \frac{s}{2})q} \big| \pa^m_t W_t f (\cdot) \big|^q \, \frac{\di t}{t} \Big)^{1/q} | L_p (\rn) \Big\|,
\end{\eq}
\begin{\eq}   \label{2.85}
\Big\| \Big( \sum^\infty_{j=-\infty} 2^{jsq}  \big| \big( \vp^j \wh{f}\, \big)^\vee (\cdot) \big|^q \Big)^{1/q}  | L_p (\rn) \Big\|
\end{\eq}
are equivalent {\em domestic} quasi-norms in $\Fas (\rn)$, and
\begin{\eq}  \label{2.86}
\Big( \int^\infty_0 t^{(m- \frac{s}{2})q} \big\| \pa^m_t W_t f \, | L_p (\rn) \big\|^q \, \frac{\di t}{t} \Big)^{1/q},
\end{\eq}
\begin{\eq}  \label{2.87}
\Big( \sum^\infty_{j=-\infty} 2^{jsq} \big\| \big( \vp^j \wh{f}\, \big)^\vee | L_p (\rn) \big\|^q \Big)^{1/q}
\end{\eq}
are equivalent {\em domestic} quasi-norms in $\Bas (\rn)$ $($usual modification if $q=\infty)$.
\\[0.1cm]
{\upshape (ii)} Let, in addition, $s<0$. Then \eqref{2.24} are  equivalent {\em admissible} quasi-norms in $\Bas (\rn)$ and \eqref{2.25} are equivalent
{\em admissible} quasi-norms in $\Fas (\rn)$. Furthermore, 
\begin{\eq}  \label{2.88}
S(\rn) \hra \Aas (\rn) \hra \As (\rn) \hra S'(\rn).
\end{\eq}
{\upshape (iii)} Let  $\Aas (\rn)$,
\begin{\eq}   \label{2.89}
A=F \quad \text{with} \quad
\begin{cases}
0<p<\infty, \ 0<q\le \infty, &\sigma_p <s< \frac{n}{p},\\
0<p\le 1, \ \ \; 0<q\le \infty, &s = \frac{n}{p},
\end{cases}
\end{\eq}
and
\begin{\eq}   \label{2.90}
A=B \quad \text{with} \quad
\begin{cases}
0<p<\infty, \ 0<q\le \infty, &\sigma_p <s< \frac{n}{p}, \\
0<p<\infty, \ q=\infty,      &\sigma_p =s, \\
0<p<\infty, \ 0<q\le 1,      &s= \frac{n}{p}.
\end{cases}
\end{\eq}
Then
\begin{\eq}   \label{2.91}
S(\rn) \hra \As (\rn) \hra \Aas (\rn) \hra \Ca{}^{-n/r} (\rn) \hra S'(\rn)
\end{\eq}
with $L_\infty (\rn)$ in place of $\Ca{}^{-n/r} (\rn)$ if $s= n/p$.
\end{theorem}

\begin{proof}
{\em Step 1.} The above assertions with $n(\frac{1}{p} - 1) <s< \frac{n}{p}$ are covered by \cite[Theorem 3.24, pp.\,79/80]{T15} complemented by
\cite[Theorem 3.5, p.\,52]{T15} as far as $p=\infty$, $s<0$, is concerned. We have to care for the new limiting cases. Some arguments remain unchanged
and will not be repeated. This applies to the Fatou property of all spaces including the spaces with $s=n/p$. 
 the homogeneity \eqref{2.83}, and the equivalence of the domestic 
quasi-norms \eqref{2.84}, \eqref{2.85} in all admitted spaces $\Fas (\rn)$ and the domestic quasi-norms \eqref{2.86}, \eqref{2.87} in all admitted 
spaces $\Bas (\rn)$ (after the other properties have been established).
\\[0.1cm]
{\em Step 2.} Let $1<p \le \infty$ and $s= n (\frac{1}{p} - 1)$. Then
\begin{\eq}   \label{2.92}
S(\rn) \hra \Ba{}^s_{p,\infty} (\rn) \hra B^s_{p,\infty} (\rn) \hra S'(\rn)
\end{\eq}
follows from the Theorems \ref{T2.4} and \ref{T2.6}. If $0<p \le 1$ and $s= n(\frac{1}{p} -1) = \sigma_p$ then one has by \eqref{2.70} and \eqref{2.76}
\begin{\eq}  \label{2.93}
\| f \, | \Ba{}^s_{p,\infty} (\rn) \|_m \le c \, \| f \, |\Ca{}^{-n} (\rn) \| + c\, \| f\, | B^s_{p,\infty} (\rn) \|.
\end{\eq}
One obtains by \eqref{2.40} with $\Ca{}^{-n} (\rn)$ and a corresponding assertion for the inhomogeneous spaces $B^s_{p,\infty} (\rn)$ that $S(\rn) 
\hra \Ba{}^s_{p,\infty} (\rn)$. Combined with
\begin{\eq}   \label{2.94}
\| f \, | \Ca{}^{-n} (\rn) \| \le c \, \sup_{t>0} t^{m - \frac{s}{2}} \| \pa^m_t  W_t f \, | L_p (\rn) \|, \quad f \in \Ba{}^s_{p,\infty} (\rn)
\end{\eq}
(as a consequence of a related embedding based on the domestic quasi-norms \eqref{2.86}, \eqref{2.87}) one obtains again by \eqref{2.76}
\begin{\eq}   \label{2.95}
\| f \, | \Ba{}^s_{p,\infty} (\rn) \| \le c \, \| f \, | B^s_{p,\infty}  (\rn) \|, \qquad f\in S(\rn).
\end{\eq}
All spaces have the Fatou property. We comment on the use of the Fatou property in Remark \ref{R2.11} below. Then one can extend \eqref{2.95} to
\begin{\eq}   \label{2.96}
S(\rn) \hra B^s_{p,\infty} (\rn) \hra \Ba{}^s_{p, \infty} (\rn) \hra \Ca{}^{-n} (\rn) \hra S' (\rn).
\end{\eq}
This proves \eqref{2.91} for $B^s_{p,\infty} (\rn)$, $0<p \le 1$, $s=\sigma_p$. These arguments can be extended to $1<p<\infty$, $s=\sigma_p =0$
resulting in
\begin{\eq}   \label{2.97}
S(\rn) \hra B^0_{p,\infty} (\rn) \hra \Ba{}^0_{p,\infty} (\rn) \hra \Ca{}^{-n/p} (\rn) \hra S' (\rn).
\end{\eq}
Both \eqref{2.96}, \eqref{2.97} justify \eqref{2.91} for the spaces in \eqref{2.90} with $s= \sigma_p$.
\\[0.1cm]
{\em Step 3.} In \cite[Sections 3.3, 3.7]{T15} we based the theory of the spaces $\Aas (\rn)$ with $\sigma_p <s<n/p$ on embeddings of type
\begin{\eq}   \label{2.98}
\Fs (\rn) \hra L_r (\rn), \qquad s- \frac{n}{p} = -\frac{n}{r},
\end{\eq}
that is $1<r<\infty$, \cite[(3.83), (3.84), (3.148), (3.149), pp.\,59, 72]{T15}. The related arguments can be extended to spaces $A^{n/p}_{p,q} (\rn)$
covered by \eqref{2.78}, \eqref{2.79}. In particular,
\begin{\eq}   \label{2.99}
S(\rn) \hra A^{n/p}_{p,q} (\rn) \hra \Aa{}^{n/p}_{p,q} (\rn) \hra L_\infty (\rn) \hra S'(\rn).
\end{\eq}
\end{proof}

\begin{remark}   \label{R2.11}
We comment on how to use the Fatou property as described in \eqref{2.30}, \eqref{2.31} in the context of the spaces $A(\rn) = \Aas (\rn)$ covered
by Theorem \ref{T2.10} and their inhomogeneous counterparts. Let $\vp_0$ be as in \eqref{2.3}. Let $f\in A(\rn)$ and $f^J = \vp_0 (2^{-J \cdot}) f$.
By the domestic quasi-norms \eqref{2.85}, \eqref{2.87} one has $f^J \in A(\rn)$,
\begin{\eq}   \label{2.100}
\sup_{J\in \nat} \| f^J \, | A(\rn) \| <\infty \qquad \text{and} \qquad f^J \to f \ \text{in} \ S'(\rn).
\end{\eq}
We justified in \cite[p.\,54]{T15} with a reference to \cite[p.\,49]{T83} that one can replace $f^J$ by suitable functions $g^J \in S(\rn)$ (with or
without $\supp \wh{g^J}$ compact). Finally one can replace $g^J$ by suitable functions $h^J \in D(\rn)$. In other words assertions for $A(\rn)$ first
proved for functions belonging to $D(\rn)$ or $S(\rn)$ can be extended to $A(\rn)$ if the Fatou property can be applied. In particular, the extension
of \eqref{2.95} from $S(\rn)$ to $B^s_{p,\infty} (\rn)$ and $\Ba{}^s_{p,\infty} (\rn)$ may be considered as a typical example. If $p<\infty$, $q<\infty$
for the spaces in Theorem \ref{T2.10} then $D(\rn)$ and $S(\rn)$ are dense. This is covered by \cite[Theorem 3.24, pp.\,79-81]{T15} complemented by the
newly incorporated limiting spaces. In these cases one can simply argue by completion.
\end{remark}

\begin{remark}   \label{R2.12}
The incorporation of the limiting cases with $s= n(\frac{1}{p} - 1)$ and $s = \frac{n}{p}$ in comparison with \cite{T15} might be of interest for its
own sake. But this may be also of some use in connection with Navier-Stokes equations and, even more, Keller-Segel equations as discussed in the
Introduction. In particular the spaces in \eqref{1.11} are critical for PDE models of chemotaxis. If $n=2$  (considered as the most interesting case from a biological point of view for the initial data $u_0$ in \eqref{1.9}) then $s= -2 + \frac{n}{p} = n (\frac{1}{p} - 1)$ is just a limiting 
situation according to the above theorem. Also the cases $A^{n/p}_{p,q} (\rn)$ and $\Aa{}^{s/p}_{p,q} (\rn)$ as solution spaces for $u$ in \eqref{1.4},
\eqref{1.5} (with a preference of $n=3$) and for $u$ in \eqref{1.7}--\eqref{1.10} (with a preference of $n=2$) seem to be of interest in connection
with Navier-Stokes equations and equations of Keller-Segel type.
\end{remark}

\section{Properties}   \label{S3}
\subsection{Embeddings}   \label{S3.1}
We dealt in \cite{T15} with some properties of the spaces $\Aas (\rn)$ as introduced in Definition \ref{D2.8}(i). Some extensions to the limiting spaces
according to Definition \ref{D2.8}(ii)--(iv) are possible without additional efforts, but will not be discussed. We are mostly interested in new 
properties subject of the subsequent sections. The present Section \ref{S3.1} is an exception. Here we are dealing with embeddings between tempered
homogeneous spaces $\Aas (\rn)$ and compare them with related assertions for their inhomogeneous counterparts. In principle one could adopt the following point of view: One can ask of whether proofs of assertions for inhomogeneous spaces $\As (\rn)$ based on the Fourier-analytical definitions
\eqref{2.7}, \eqref{2.9} can be transferred to the homogeneous counterparts \eqref{2.85}, \eqref{2.87}. This works quite often but not always. But our
aim here is different. We prove some more sophisticated embeddings via heat kernels and demonstrate how to use that the spaces $\Aas (\rn)$ with $s>0$
coincide locally with $\As (\rn)$ (as already done in \cite{T15}). Recall our abbreviation
\begin{\eq}  \label{3.0a}
\Cas (\rn) = \Ba{}^s_{\infty, \infty} (\rn) \quad \text{and} \quad \Cc^s (\rn) = B^s_{\infty, \infty} (\rn), \quad -n\le s <0,
\end{\eq}
normed by
\begin{\eq}  \label{3.0b}
\| f \, | \Cas (\rn) \| = \sup_{x\in \rn, t>0} t^{-s/2} |W_t f(x)|, \quad \|f\, | \Cc^s (\rn) \| = \sup_{x\in \rn, 0<t<1} t^{-s/2} |W_t f(x)|,
\end{\eq}
\eqref{2.29}, \eqref{2.21}.   

\begin{theorem}   \label{T3.1}
Let $n\in \nat$.
\\[0.1cm]
{\upshape (i)} Let $-n<s<0$ and $0<q_0 <q_1 <q_2 \le \infty$. Then
\begin{\eq}   \label{3.1}
\Ba{}^s_{\infty,q_0} (\rn) \hra \Ba{}^s_{\infty, q_1} (\rn) \hra \Fa{}^s_{\infty, q_1} (\rn) \hra \Fa{}^s_{\infty, q_2} (\rn) \hra \Cas (\rn)
\end{\eq}
and
\begin{\eq}   \label{3.2}
B^s_{\infty,q_0} (\rn) \hra B^s_{\infty, q_1} (\rn) \hra F^s_{\infty, q_1} (\rn) \hra F^s_{\infty, q_2} (\rn) \hra \Cc^s (\rn).
\end{\eq}
{\upshape (ii)} Let $-n <s<0$, $0<p<\infty$ and $0<q\le \infty$. Then
\begin{\eq}   \label{3.3}
\Ba{}^{s+ \frac{n}{p}}_{p,\infty} (\rn) \hra \Fa{}^s_{\infty,q} (\rn) \quad \text{and} \quad B^{s+ \frac{n}{p}}_{p,\infty} (\rn) \hra F^s_{\infty,q}
(\rn).
\end{\eq}
{\upshape (iii)} Let $0<p<\infty$, $0<q \le \infty$ and $n(  \frac{1}{p} - 1) <s< \frac{n}{p}$. Let
\begin{\eq}   \label{3.4}
0<p_0 <p < p_1 \le \infty \quad \text{and} \quad s_0 - \frac{n}{p_0} = s - \frac{n}{p} = s_1 - \frac{n}{p_1}.
\end{\eq}
Then
\begin{\eq}   \label{3.5}
\Ba{}^{s_0}_{p_0,u} (\rn) \hra \Fas (\rn) \hra \Ba{}^{s_1}_{p_1,v} (\rn)
\end{\eq}
if,  and only if, $0<u\le p \le v \le \infty$, and
\begin{\eq}   \label{3.6}
B^{s_0}_{p_0,u} (\rn) \hra \Fs (\rn) \hra B^{s_1}_{p_1,v} (\rn)
\end{\eq}
if, and only if, $0<u\le p \le v \le \infty$.
\end{theorem}

\begin{proof}
{\em Step 1.} We prove part (i). Let $f\in \Ba{}^s_{\infty,q_1} (\rn)$. Then one has by \eqref{2.26} and \eqref{2.24} with $p=\infty$
\begin{\eq}  \label{3.7}
\begin{aligned}
\| f \, | \Fa{}^s_{\infty,q_1} (\rn) \| &\le c \, \sup_{x\in \rn, t>0} \Big( \int^t_0 \tau^{-sq_1/2} \sup_{z\in \rn} |W_\tau f(z)|^{q_1} \frac{\di 
\tau}{\tau} \Big)^{1/q_1} \\
&= c \, \| f \, | \Ba{}^s_{\infty, q_1} (\rn) \|.
\end{aligned}
\end{\eq}
This proves the second embedding in \eqref{3.1}.
The proof of the last embedding in \eqref{3.1} with $q_2 <\infty$ relies on the sub-mean value property of the heat equation
\begin{\eq}   \label{3.8}
\big| W_t f(x) \big|^{q_2} \le c \, t^{- \frac{n}{2} -1} \int^t_{t/2} \int_{|x-y| \le \sqrt{t}} \big| W_\tau f(y) \big|^{q_2} \, \di y \, \di \tau,
\quad x \in \rn, \quad t>0,
\end{\eq}
\cite[(3.16)]{T15} with a reference of \cite[Lemma 2, p.\,172]{Bui83}.  One has
\begin{\eq}   \label{3.9}
\begin{aligned}
t^{-s/2} \big| W_t f(x)\big| &\le c \, \Big( t^{-n/2} \int^t_0 \int_{|x-y|\le \sqrt{t}} \tau^{-sq_2/2} \big| W_\tau f(y) \big|^{q_2} \, \di y \, \frac{\di \tau}{\tau} \Big)^{1/q_2} \\
&\le c \, \| f \, | \Fa{}^s_{\infty,q_2} (\rn) \|
\end{aligned}
\end{\eq}
where we used again \eqref{2.26}.  Then the last embedding in \eqref{3.1} with $q_2 <\infty$ follows from \eqref{3.0b}. If $q_2 =\infty$ then one has
\eqref{2.28}, where the indicated short proof with a reference to \cite[p.\,47]{T15} is just a modification of \eqref{3.8}, \eqref{3.9}. Let $f\in
\Fa{}^s_{\infty,q_1} (\rn)$. By \eqref{2.26}, \eqref{3.0b} and what we already know the last but one inequality in \eqref{3.1} follows from
\begin{\eq}   \label{3.10}
\begin{aligned}
\| f \, |\Fa{}^s_{\infty, q_2} (\rn) \| & \le \Big( \sup_{\tau >0, y \in \rn} \tau^{-s(q_2 - q_1)/2} \big| W_\tau f(y) \big|^{q_2-q_1} \Big)^{1/q_2} \\
&\quad \times \sup_{x\in \rn, t>0} \Big(t^{-n/2} \int^t_0 \int_{|x-y| \le \sqrt{t}} \tau^{- sq_1/2} \big| W_\tau f(y) \big|^{q_1} \, \di y \, 
\frac{\di \tau}{\tau} \Big)^{1/q_2} \\
&\le c \, \| f\, | \Cas (\rn) \|^{1- \frac{q_1}{q_2}} \, \| f \, | \Fa{}^s_{\infty, q_1} (\rn)  \|^{\frac{q_1}{q_2}} \\
& \le c' \, \| f \, | \Fa{}^s_{\infty, q_1} (\rn) \|.
\end{aligned}
\end{\eq}
Based on \eqref{2.24} and \eqref{3.0b} one can argue similarly to prove the first embedding in \eqref{3.1},
\begin{\eq}   \label{3.11}
\begin{aligned}
\|f \, | \Ba{}^s_{\infty,q_1} (\rn) \| & \le c \, \| f\, | \Cas (\rn)\|^{1- \frac{q_0}{q_1}} \Big( \int^\infty_0 \tau^{-q_0/2} \sup_{z \in \rn}
\big| W_\tau f(z) \big|^{q_0} \, \frac{\di \tau}{\tau} \Big)^{1/q_1} \\
&\le c' \, \| f \, | \Ba{}^s_{\infty,q_0} (\rn) \|.
\end{aligned}
\end{\eq}
This completes the proof of \eqref{3.1}. In the same way one can prove the inhomogeneous counterpart \eqref{3.2} using \eqref{2.16}--\eqref{2.18} and
\eqref{3.0b}.
\\[0.1cm]
{\em Step 2.} We prove \eqref{3.3}. By \eqref{3.1} we may assume $0<q<p< \infty$. Using the domestic equivalent quasi-norms \eqref{2.87} one obtains
\begin{\eq}   \label{3.12}
\Ba{}^{s+ \frac{n}{p_1}}_{p_1,q} (\rn) \hra \Ba{}^{s+ \frac{n}{p_2}}_{p_2,q} (\rn), \qquad 0<p_1 <p_2 \le \infty,
\end{\eq}
in the same way as in the inhomogeneous case, \cite[Theorem 2.7.1, p.\,129]{T83}. This shows that in addition to $0<q<p<\infty$ one may assume $-n <s<
s + \frac{n}{p} <0$. We rely again on \eqref{2.26}, \eqref{2.24} and assume $f\in \Ba{}^{s+ \frac{n}{p}}_{p, \infty} (\rn)$. One has by H\"{o}lder's
inequality
\begin{\eq}   \label{3.13}
\begin{aligned}
&t^{-\frac{n}{2}} \int^t_0 \int_{|x-y| \le \sqrt{t}} \tau^{- \frac{sq}{2}} | W_\tau f(y) |^q \, \di y \frac{\di \tau}{\tau} \\
& \qquad \le c\, \int^t_0  \tau^{- \frac{sq}{2}} \Big( t^{- \frac{n}{2}} \int_{|x-y| \le \sqrt{t}} |W_\tau f(y) |^p \, \di y \Big)^{\frac{q}{p}} 
\frac{\di \tau}{\tau} \\
& \qquad \le c' \sup_{\sigma >0} \sigma^{- \frac{q}{2} (s+ \frac{n}{p})} \, \| W_\sigma f \, | L_p (\rn) \|^q \, t^{-\frac{nq}{2p}} \int^t_0 
\tau^{\frac{nq}{2p}} \, \frac{\di \tau}{\tau} \\
&\qquad \le c'' \, \| f \, | \Ba{}^{s+\frac{n}{p}}_{p,\infty} (\rn) \|^q.
\end{aligned}
\end{\eq}
This proves the first embedding in \eqref{3.3}. Similarly one justifies the second embedding based on \eqref{2.16}, \eqref{2.18}.
\\[0.1cm]
{\em Step 3}. The sharp embedding \eqref{3.6} for inhomogeneous spaces goes back to \cite{SiT95}, based on \cite{Jaw77, Fra86}. One may also consult
\cite[p.\,44]{ET96} and the additional references within. If $0<p<\infty$, $0<q\le \infty$ and $\sigma_p <\sigma < n/p$ then the spaces $A^\sigma_{p,q} (\rn)$ and $\Aa{}^\sigma_{p,q} (\rn)$ coincide locally,
\begin{\eq}   \label{3.14}
\| g \, | \Aa{}^\sigma_{p,q} (\rn) \| \sim \| g \, | A^\sigma_{p,q} (\rn) \|, \qquad g\in S' (\rn), \quad \supp g \subset \ol{\Om},
\end{\eq}
with, say, $\Om = \{ y: \ |y| <1 \}$. This is covered by \cite[Proposition 3.52, p.\,107]{T15}. Then one has for some $c_1 >0$, $c_2 >0$
\begin{\eq}   \label{3.15}
c_1 \, \| f \, | \Ba{}^{s_1}_{p_1,v} (\rn) \| \le \|f\, | \Fas (\rn) \| \le c_2 \, \| f \, | \Ba{}^{s_0}_{p_0,u} (\rn) \|
\end{\eq}
if, in addition $s_1 >0$ (then also $s>0$, $s_0 >0$), and $\supp f  \subset \ol{\Om}$. All spaces in \eqref{3.15} have the same homogeneity $s - 
\frac{n}{p}$ according to \eqref{2.83}. Then one can extend \eqref{3.15} by homogeneity to all compactly supported functions $f$. The rest is now a
matter of the Fatou property of the underlying spaces as explained in Remark \ref{R2.11} resulting in \eqref{3.5} so far for $s_1 >0$. By 
\cite[Proposition 3.41, p.\,97]{T15}
\begin{\eq}   \label{3.16}
\dot{I}_\sigma \Aas (\rn) = \Aa{}^{s-\sigma}_{p,q} (\rn), \qquad \dot{I}_\sigma f = \big( |\xi|^\sigma \wh{f} \big)^\vee,
\end{\eq}
is a lift within the distinguished strip. This extends \eqref{3.5} to all spaces which can be reached in this way, based on what we already know. This
covers all spaces with $p_1 <\infty$. If $p_1 =\infty$ and $p<p_2 <\infty$, $ s_2 - \frac{n}{p_2} =s_1$, then one has by \eqref{3.5}  with 
$\Ba{}^{s_2}_{p_2,v} (\rn)$ on the right-hand side and \eqref{3.12}
\begin{\eq}   \label{3.17}
\Fas (\rn) \hra \Ba{}^{s_2}_{p_2,v} (\rn) \hra \Ba{}^{s_1}_{\infty, v} (\rn), \qquad p\le v.
\end{\eq}
As for the sharpness of
\begin{\eq}   \label{3.18}
\Fas (\rn) \hra \Ba{}^{s_1}_{\infty, v} (\rn), \qquad p\le v
\end{\eq}
we may assume by lifting $s>0$. Recall $s_1 <0$. Then one obtains by \eqref{2.34} and \eqref{3.14}
\begin{\eq}   \label{3.19}
\Fs (\rn) \hra \Ba{}^{s_1}_{\infty, v} (\rn) \hra B^{s_1}_{\infty, v} (\rn)
\end{\eq}
if $\supp f \subset \ol{\Om}$. But the embedding \eqref{3.6} and its sharpness is a local matter. This means that \eqref{3.19} requires $p\le v$. 
This proves part (iii).
\end{proof}

\begin{remark}   \label{R3.2}
The embedding \eqref{3.3} for the inhomogeneous spaces is known and goes back to \cite[Lemma 16, p.\,253]{Mar95} (with a short Fourier-analytical 
proof). 
\end{remark}

\subsection{Global inequalities for heat equations}   \label{S3.2}
Let $\As (\rn)$ be the inhomogeneous spaces as introduced in Section \ref{S2.1} and let $W_t w$ be given by \eqref{2.12}, \eqref{2.13}. Let $1\le p,q
\le \infty$ ($p<\infty$ for $F$-spaces), $s\in \real$ and $d\ge 0$. According to \cite[Theorem 4.1, p.\,114]{T14} there is a constant $c>0$ such that
for all $t$ with $0<t \le 1$ and all $w\in \As (\rn)$,
\begin{\eq}   \label{3.20}
t^{d/2} \| W_t w \, | A^{s+d}_{p,q} (\rn) \| \le c\, \| w \, | \As (\rn) \|.
\end{\eq}
We ask for a counterpart in terms of the tempered homogeneous spaces $\Aas (\rn)$ according to Definition \ref{D2.8}(i) in the distinguished strip
\eqref{2.67}.

\begin{theorem}   \label{T3.3}
Let $n\in \nat$. Let $1<p<\infty$, $1\le q \le \infty$ and
\begin{\eq}   \label{3.21}
n \big( \frac{1}{p} - 1 \big) <s<0< s+d < \frac{n}{p}.
\end{\eq}
Then there is constant $c>0$ such that for all $t$, $t>0$, and all $w \in \Aas(\rn)$,
\begin{\eq}   \label{3.22}
t^{d/2} \, \| W_t w \, | \Aa{}^{s+d}_{p,q} (\rn) \| \le c \, \| w \, | \Aas (\rn) \|.
\end{\eq}
\end{theorem}

\begin{proof}
By \eqref{3.20} on the one hand and \eqref{2.88}, \eqref{2.91} on the other hand one has
\begin{\eq}   \label{3.23}
t^{d/2} \| W_t w \, | \Aa{}^{s+d}_{p,q} (\rn) \| \le c \, \|w \, | \Aas (\rn) \|, \qquad 0<t \le 1.
\end{\eq}
Furthermore,
\begin{\eq}   \label{3.24}
\big( W_{t \lambda^2} w \big)(\lambda x) = W_t w(\lambda \cdot) (x), \qquad x\in \rn, \quad t>0, \quad \lambda >0,
\end{\eq}
is an immediate consequence of \eqref{2.12}, \eqref{2.13}. We replace $w$ in \eqref{3.23} by $w(\lambda \cdot)$, $\lambda >0$. Using the homogeneity
\eqref{2.83} one obtains
\begin{\eq}  \label{3.24a}
t^{d/2} \lambda^{s+d-\frac{n}{p}} \| W_{t \lambda^2} w \, | \Aa{}^{s+d}_{p,q} (\rn) \| \le c\, \lambda^{s- \frac{n}{p}} \| w \, | \Aas (\rn) \|.
\end{\eq}
This proves \eqref{3.22} for all $w\in \Aas (\rn)$ and all $t>0$.
\end{proof}

\begin{remark}   \label{R3.4}
One can extend the above theorem to the limiting cases covered by Theorem \ref{T2.10}. Of special interest may be the spaces with $s+d = n/p$. In
particular, if
\begin{\eq}   \label{3.25}
1<p<\infty, \quad n \big( \frac{1}{p} -1 \big) <s<0 \quad \text{and} \quad s+d = \frac{n}{p},
\end{\eq}
then
\begin{\eq}   \label{3.26}
t^{\frac{n}{2p}} \, \| W_t w \, | \Ba{}^{\frac{n}{p}}_{p,1} (\rn) \| \le c \, t^{\frac{s}{2}} \, \| w \, | \Ba{}^s_{p,1} (\rn) \|
\end{\eq} 
for all $w\in \Ba{}^s_{p,1} (\rn)$ and all $t>0$. The special case $n<p<\infty$ and $s=-1 + \frac{n}{p}$ may be of special interest in the
context of Navier-Stokes equations, that is
\begin{\eq}   \label{3.26a}
\sqrt{t} \ \| W_t w \, | \Ba{}^{\frac{n}{p}}_{p,1} (\rn) \| \le c \, \| w \, | \Ba{}^{\frac{n}{p} - 1}_{p,1} (\rn) \|
\end{\eq}
for all $w\in \Ba{}^{\frac{n}{p} -1}_{p,1} (\rn)$ and all $t>0$.
\end{remark}

\subsection{Hardy inequalities}   \label{S3.3}
We have a closer look at some spaces $\Aas(\rn)$ according to Definition \ref{D2.8}(i) consisting entirely of regular tempered distributions. Let
$0<p<\infty$, $0<q\le \infty$ and
\begin{\eq}   \label{3.27}
\sigma_p = \max \Big( 0, n \big( \frac{1}{p} - 1 \big) \Big) <s< \frac{n}{p}, \qquad - \frac{n}{r} = s - \frac{n}{p}.
\end{\eq}
Then $1<r<\infty$. Let $s/2 <m\in \nat$. Then $\Fas(\rn)$ is the collection of all regular tempered distributions $f\in S'(\rn) \cap \Lloc (\rn)$ such
that
\begin{\eq}   \label{3.28}
\Big\| \Big( \int^\infty_0 t^{(m- \frac{s}{2})q} \big| \pa^m_t W_t f(\cdot) \big|^q \frac{\di t}{t} \Big)^{1/q} | L_p (\rn) \Big\| + \| f \, | L_r (\rn) \|
\end{\eq}
is finite. This may be interpreted as an equivalent {\em domestic} quasi-norm in $\Fas (\rn)$ as introduced in Definition \ref{D2.8}(i) if $s>0$.
Furthermore,
\begin{\eq}   \label{3.29}
\| f \, |L_r (\rn) \| \le c\, \Big\| \Big( \int^\infty_0 t^{(m- \frac{s}{2})q} \big| \pa^m_t W_t f(\cdot) \big|^q \, \frac{\di t}{t} \Big)^{1/q} |
L_p (\rn) \Big\|,
\end{\eq}
$f\in \Fas (\rn)$. Similarly for $\Bas (\rn)$ under the additional restriction $0<q\le r$. Then
\begin{\eq}   \label{3.30}
\| f \, | L_r (\rn) \| \le c \, \Big( \int^\infty_0 t^{(m- \frac{s}{2})q} \big\| \pa^m_t W_t f \, | L_p (\rn) \big\|^q \frac{\di t}{t} \Big)^{1/q},
\end{\eq}
$f \in \Bas (\rn)$. Details may be found in \cite[Section 3.3, pp.\,57--62]{T15}. But instead of $L_r$, having the same homogeneity $-\frac{n}{r} =
s - \frac{n}{p}$ as $\Aas (\rn)$ one can use weighted $L_p$-spaces and weighted $L_q$-spaces with the same homogeneity. For this purpose we recall
first the {\em Hardy inequalities}  for the corresponding inhomogeneous spaces $\As(\rn)$.

Let $p,q,s$ and $r$ be as as above. Then there is a constant $c>0$ such that
\begin{\eq}   \label{3.31}
\int_{|x| \le 1} \big| \, |x|^{\frac{n}{r}} f(x) \big|^p \frac{\di x}{|x|^n} \le c \, \| f \, | \Fs (\rn) \|^p
\end{\eq}
for all $f\in \Fs (\rn)$. Let, in addition, $0<q\le r$. Then there is a constant $c>0$ such that
\begin{\eq}   \label{3.32}
\int_{|x| \le 1} \big| \, |x|^{\frac{n}{r}} f(x) \big|^q \, \frac{\di x}{|x|^n} \le c \, \| f \, | \Bs (\rn) \|^q
\end{\eq}
for all $f\in \Bs (\rn)$. Details and further explanations may be found in \cite[Theorem 16.3, p.\,238]{T01}. We extend these assertions to the related
tempered homogeneous spaces. Let $\sigma_p$ be as in \eqref{3.27}.

\begin{theorem}   \label{T3.5}
Let $0<p<\infty$,
\begin{\eq}   \label{3.33}
\sigma_p <s< \frac{n}{p}, \qquad -\frac{n}{r} = s - \frac{n}{p}.
\end{\eq}
{\upshape (i)} Let $0<q\le \infty$. Then there is a constant $c>0$ such that
\begin{\eq}   \label{3.34}
\int_{\rn} \big| \, |x|^{\frac{n}{r}} f(x) \big|^p \, \frac{\di x}{|x|^n} \le c\, \| f \, | \Fas (\rn) \|^p
\end{\eq}
for all $f\in \Fas (\rn)$.
\\[0.1cm]
{\upshape (ii)} Let $0<q\le r$. Then there  is a constant $c>0$ such that
\begin{\eq}   \label{3.35}
\int_{\rn} \big| \, |x|^{\frac{n}{r}} f(x) \big|^q \, \frac{\di x}{|x|^n} \le c \, \| f\, | \Bas (\rn) \|^q
\end{\eq}
for all $f\in \Bas (\rn)$.
\end{theorem}

\begin{proof}
Using \eqref{3.31} and \eqref{3.14} one has
\begin{\eq}   \label{3.36}
\int_{\rn} \big| \, |x|^{\frac{n}{r}} f(x) \big|^p \, \frac{\di x}{|x|^n} \le c \, \|f \, | \Fas (\rn) \|^p
\end{\eq}
for all $f\in \Fas (\rn)$ with $\supp f \subset \{y: \, |y| \le 1\}$. Let $\lambda >1$ and $f\in \Fas (\rn)$ with $\supp f \subset \{ y: \, |y| \le 
\lambda\}$. We insert $f(\lambda \cdot)$ in \eqref{3.36}. By the homogeneity \eqref{2.83} and an obvious counterpart of the left-hand side of \eqref{3.36}
with the same homogeneity one can extend \eqref{3.36} to all compactly supported $f\in \Fas (\rn)$. The rest is now a matter of the Fatou property as
explained in Remark \ref{R2.11}. This proves \eqref{3.34}. The proof of \eqref{3.35} is similar.
\end{proof}

\begin{remark}   \label{R3.6}
Let $p,s$ be as in \eqref{3.33} and $0<q\le \infty$. Let $s/2 <m\in \nat$. Then both \eqref{3.28} and
\begin{\eq}   \label{3.37}
\Big\| \Big( \int^\infty_0 t^{(m- \frac{s}{2})q} \big| \pa^m_t W_t f(\cdot) \big|^q \, \frac{\di t}{t} \Big)^{1/q} | L_p (\rn) \Big\| + \Big( \int_{\rn}
\big| \, |x|^{\frac{n}{r}} f(x) \big|^p \, \frac{\di x}{|x|^n} \Big)^{1/p}
\end{\eq}
are equivalent {\em domestic} quasi-norms in $\Fas (\rn)$. Furthermore
\begin{\eq}   \label{3.38}
\Big( \int_{\rn}\big| \, |x|^{\frac{n}{r}} f(x) \big|^p \, \frac{\di x}{|x|^n} \Big)^{1/p} \le c\,
\Big\| \Big( \int^\infty_0 t^{(m- \frac{s}{2})q} \big| \pa^m_t W_t f(\cdot) \big|^q \, \frac{\di t}{t} \Big)^{1/q} | L_p (\rn) \Big\| 
\end{\eq}
and 
\begin{\eq}  \label{3.39}
\Big( \int_{\rn}\big| \, |x|^{\frac{n}{r}} f(x) \big|^p \, \frac{\di x}{|x|^n} \Big)^{1/p} \le c\,
\Big\| \big( \sum^\infty_{j=-\infty} 2^{jsq} \big| (\vp^j \wh{f} \, )^\vee (\cdot) \big|^q \Big)^{1/q} | L_p (\rn) \Big\|
\end{\eq}
for some $c>0$ and all
$f\in \Fas (\rn)$. Here $\{ \vp^j \}$ is the usual dyadic resolution of unity in $\rn \setminus \{0 \}$. This is covered by Theorem \ref{T2.10}.
If, in addition, $0<q\le r$ then
\begin{\eq}   \label{3.40}
\Big( \int^\infty_0 t^{(m - \frac{s}{2})q} \big\| \pa^m_t W_t f \, | L_p (\rn) \big\|^q \, \frac{\di t}{t} \Big)^{1/q} + \Big( \int_{\rn} \big| \,
|x|^{\frac{n}{r}} f(x) \big|^q \, \frac{\di x}{|x|^n} \Big)^{1/q}
\end{\eq}
is an equivalent {\em domestic} quasi-norm in $\Bas (\rn)$. There are obvious counterparts of \eqref{3.38}, \eqref{3.39}.
\end{remark}

\subsection{Multiplication algebras}   \label{S3.4}
Let $\As (\rn)$ be the inhomogeneous spaces as introduced in \eqref{2.6}--\eqref{2.9}.
Recall that $\As (\rn)$ is called a {\em multiplication algebra} if $f_1 f_2 \in \As (\rn)$ for any $f_1 \in \As (\rn)$, $f_2 \in \As (\rn)$ and if
there is a constant $c>0$ such that
\begin{\eq}   \label{3.40a}
\| f_1 f_2 \, | \As (\rn) \| \le c \, \|f_1 \, | \As (\rn) \| \cdot \|f_2 \, | \As (\rn) \|
\end{\eq}
for all $f_1 \in \As (\rn)$, $f_2 \in \As (\rn)$. As for (technical) details we refer the reader to \cite[Section 1.2.5]{T13}, \cite[Section 
3.2.4]{T14} and the literature within. Let $0<p,q \le \infty$ ($p<\infty$ for $F$-spaces) and $s\in \real$. Then the following assertions are pairwise
equivalent:
\\[0.1cm]
(a) $\As  (\rn)$ {\em is a multiplication algebra.}
\\[0.1cm]
(b) $s>0$ {\em and} $\As  (\rn) \hra L_\infty (\rn)$.
\\[0.1cm]
(c) {\em Either}
\begin{\eq}   \label{3.41}
\As (\rn) = \Bs (\rn) \ \text{{\em with}} \
\begin{cases}
s>n/p &\text{{\em where} $0<p,q \le \infty$}, \\
s= n/p &\text{{\em where} $0<p<\infty$, $0<q \le 1$,}
\end{cases}
\end{\eq}
{\em or}
\begin{\eq}   \label{3.42}
\As (\rn) = \Fs (\rn) \ \text{{\em with}} \
\begin{cases}
s> n/p &\text{{\em where} $0<p<\infty$, $0<q\le \infty$},\\
s=n/p  &\text{{\em where} $0<p\le 1$, $0<q \le \infty$}.
\end{cases}
\end{\eq}
These assertions have a long history, related references may be found in \cite[Section 1.2.3, pp.\,12--13]{T13}. Let $\Aas (\rn)$ be the tempered
homogeneous spaces as introduced in Definition \ref{D2.8}. One may again ask which of these spaces is a multiplication algebra with
\begin{\eq}   \label{3.43}
\| f_1 f_2 \, | \Aas (\rn) \| \le c \, \| f_1 \, | \Aas (\rn) \| \cdot \| f_2 \, | \Aas (\rn) \|
\end{\eq}
as the obvious counterpart of \eqref{3.40a}. If one replaces $f_1$ by $f_1 (\lambda \cdot)$ and $f_2$ by $f_2 (\lambda \cdot)$, $\lambda >0$, then it
follows from \eqref{2.83} that
\begin{\eq}   \label{3.44}
\| f_1 f_2 \, | \Aas (\rn) \| \le c \, \lambda^{s - \frac{n}{p}} \, \| f_1 \, | \Aas (\rn) \| \cdot \|f_2 \, | \Aas (\rn) \|.
\end{\eq}
In other words, only spaces with $s= n/p$ have a chance to be multiplication algebras. This applies to the spaces according to Definition 
\ref{D2.8}(iii),(iv) with the following outcome.

\begin{theorem}   \label{T3.7}
A space $\Aas (\rn)$ as introduced in Definition \ref{D2.8} is a multiplication algebra if, and only if, either
\begin{\eq}   \label{3.45}
\Aas (\rn) = \Fas (\rn), \qquad 0<p\le 1, \quad 0< q \le \infty, \quad s=n/p
\end{\eq}
or
\begin{\eq}   \label{3.46}
\Aas (\rn) = \Bas (\rn), \qquad 0<p<\infty, \quad 0<q\le 1, \quad s=n/p.
\end{\eq}
\end{theorem}

\begin{proof}
Let $f_1, f_2 \in \Aas (\rn)$ with $\Aas (\rn)$ as in \eqref{3.45}, \eqref{3.46}. Let, in addition, $\supp f_1 \subset \ol{\Om}$, $\supp f_2 \subset
\ol{\Om}$ with $\Om = \{y: \ |y| <1 \}$. Then \eqref{3.14} can be extended to the spaces in \eqref{3.45}, \eqref{3.46}, again with a reference to
\cite[Proposition 3.52, p.\,107]{T15}, that is
\begin{\eq}   \label{3.47}
\| f_k \, | \Aa{}^{n/p}_{p,q} (\rn) \| \sim \| f_k \, | A^{n/p}_{p,q} (\rn) \|, \qquad k=1,2.
\end{\eq}
According to \eqref{3.41}, \eqref{3.42} the spaces $A^{n/p}_{p,q} (\rn)$ in question are multiplication algebras. Using \eqref{3.44} with $s=n/p$ it
follows
\begin{\eq}   \label{3.48}
\| f_1 f_2 \, | \Aas (\rn) \| \le c \, \| f_1 \, | \Aas (\rn) \| \cdot \|f_2 \, | \Aas (\rn) \|
\end{\eq}
for all compactly supported $f_1, f_2 \in \Aas (\rn)$. The rest is now a matter of the Fatou property as explained in Remark \ref{R2.11}. As mentioned
after \eqref{3.44} spaces $\Aas (\rn)$ with $s<n/p$ cannot be multiplication algebras.
\end{proof}

\begin{remark}   \label{R3.8}
In connection with possible applications to Navier-Stokes equations and Keller-Segel systems as described in Section \ref{S1} those multiplication
algebras might be of interest which are also Banach spaces (and not only quasi-Banach spaces). This applies to $\Fa{}^n_{1,q} (\rn)$, $1\le q \le
\infty$, and $\Ba{}^{n/p}_{p,1} (\rn)$, $1 \le p <\infty$. The assertion that $\Ba{}^{n/p}_{p,1} (\rn)$ is a multiplication algebra has already been
observed in \cite[p.\,148]{Pee76}, denoted there as $\dot{B}^{n/p}_{p,1} (\rn)$, considered in the context of $\big(\dot{S} (\rn), \dot{S}'(\rn) 
\big)$.
\end{remark}

\subsection{Examples I: Riesz kernels}   \label{S3.5}
We discuss some examples which illuminate the different nature of the tempered homogeneous spaces $\Aas
(\rn)$ in the framework of the dual pairing $\big( S(\rn), S'(\rn) \big)$ and their inhomogeneous counterparts $\As (\rn)$. 

We begin with a simple observation. 
Let $f(x) =1$, $x\in \rn$. By \eqref{2.12} or \eqref{2.13}
one has $W_t f(x) =1$, $t>0$, $x\in \rn$. Then it follows from \eqref{2.29} with $\Ba{}^s_{\infty, \infty}(\rn) = \Cas (\rn)$ that $f\not\in \Cas(\rn)$
 for any $s<0$. On the other hand $f\in \Cc^s (\rn) = B^s_{\infty, \infty} (\rn)$ for all  $s\in \real$.

But we are mainly interested to have a closer look at the kernels of the Riesz transform
\begin{\eq}   \label{3.49}
h_\sigma (x) = |x|^{-\sigma}, \qquad 0<\sigma <n, \quad x \in \rn.
\end{\eq}
The interest in these kernels comes from
\begin{\eq}   \label{3.50}
\wh{h_\sigma} (\xi) = c_\sigma \, |\xi|^{\sigma -n}, \qquad 0<\sigma <n, \quad c_\sigma >0, \quad 0 \not= \xi \in \rn,
\end{\eq}
\cite[pp.\,117, 73]{Ste70}, \cite[Theorem 5.9, p.\,122]{LiL97}, and the related Riesz potentials
\begin{\eq}   \label{3.51}
\big( h_\sigma \wh{f} \big)^\vee (x) = c \, \int_{\rn} \frac{f(y)}{|x-y|^{n-\sigma}} \, \di y, \qquad 0<\sigma<n, \quad x \in \rn.
\end{\eq}
Let again $\Aas (\rn)$ be the tempered homogeneous spaces in the distinguished strip as introduced in Definition \ref{D2.8}. If $h_\sigma \in \Aas(\rn)$
then one has by \eqref{2.83} for $\lambda >0$,
\begin{\eq}  \label{3.52}
\lambda^{-\sigma} \big\| \, |x|^{-\sigma} \, | \Aas (\rn) \big\| = \big\| \, |\lambda x|^{-\sigma} \, | \Aas (\rn) \big\| 
= \lambda^{s- \frac{n}{p}} \big\| \, |x|^{-\sigma} \, | \Aas (\rn) \big\|.
\end{\eq}
One obtains $s= \frac{n}{p} - \sigma$ as a necessary condition for $h_\sigma \in \Aas (\rn)$. (Recall that for fixed $p,q$ and $A\in \{B,F \}$
there is no monotonicity of $\Aas (\rn)$ with respect to $s$ in contrast to the related inhomogeneous spaces). There is a temptation to use the 
equivalent domestic quasi-norms \eqref{2.85}, \eqref{2.87} to clarify whether $h_\sigma \in \Aas (\rn)$ where $s- \frac{n}{p} = -\sigma = - \frac{n}{r}$
is covered by \eqref{2.67}. But this requires that one first ensures $h_\sigma \in \Ca{}^{-\sigma} (\rn)$ as the {\em anchor space} of $\Aas (\rn)$ with $s = \frac{n}{p} - \sigma$ which are continuously embedded in $\Ca{}^{-\sigma} (\rn)$. Furthermore we wish to compare the outcome with 
corresponding assertions for the inhomogeneous spaces $\As (\rn)$.

\begin{theorem}  \label{T3.9}
Let $n\in \nat$ and $0<\sigma <n$.
\\[0.1cm]
{\upshape (i)} Let
\begin{\eq}   \label{3.53}
0<p \le \infty, \quad 0<q\le \infty, \quad n \big( \frac{1}{p} - 1 \big) <s< \frac{n}{p},
\end{\eq}
$(p<\infty$ for $F$-spaces$)$. Then $h_\sigma \in \Aas (\rn)$ if, and only if,
\begin{\eq}   \label{3.54}
\Aas (\rn) = \Ba{}^{\frac{n}{p} - \sigma}_{p, \infty} (\rn).
\end{\eq}
{\upshape (ii)} Let
\begin{\eq}   \label{3.55}
0<p \le \infty, \quad 0<q\le \infty, \quad s \in \real,
\end{\eq}
$(p<\infty$ for $F$-spaces$)$. Then $h_\sigma \in \As (\rn)$  if, and only if, $n/p <\sigma <n$ and
\begin{\eq}  \label{3.56}
\begin{cases}
either &\As (\rn) = B^{\frac{n}{p} - \sigma}_{p,\infty} (\rn), \\
or     &0<q \le \infty, \ s< \frac{n}{p} - \sigma.
\end{cases}
\end{\eq}
\end{theorem} 

\begin{proof}
{\em Step 1.} Let again $L_{r, \infty} (\rn)$ with $r= n/\sigma$ be the usual Lorentz (Marcinkiewicz) space quasi-normed by
\begin{\eq}   \label{3.57}
\|f \, | L_{r,\infty} (\rn) \| = \sup_{t>0} t^{1/r} f^* (t)
\end{\eq}
where $f^* (t)$ is the decreasing rearrangement of the Lebesgue-measurable function $f$ in \rn.
There are two positive constants $c, c'$ such that for any $R >0$,
\begin{\eq}   \label{3.58}
\big| \big\{ x\in \rn: \ h_\sigma (x) \ge R^{-\sigma} \big\} \big| = c\, R^n =t= \big| \big\{ \tau >0: \ h^*_\sigma (\tau) \ge c'\, t^{-\sigma/n} \big\}
\big|.
\end{\eq}
Then $h^*_\sigma (t) = c' \, t^{-\sigma/n}$ and $h_\sigma \in L_{r,\infty} (\rn)$. Let $1<r_1 <r< r_2 <\infty$. By \eqref{3.5} one has
\begin{\eq}   \label{3.59}
L_{r_j} (\rn) = \Fa{}^0_{r_j, 2} (\rn) \hra \Ca{}^{-\sigma_j} (\rn), \qquad \sigma_j = n/r_j,
\end{\eq}
$j= 1,2$. With $\frac{1}{r} = \frac{1-\theta}{r_1} + \frac{\theta}{r_2}$, $\sigma = (1-\theta)\sigma_1 + \theta \sigma_2 = \frac{n}{r}$, one has by
real interpolation
\begin{\eq}   \label{3.60}
L_{r, \infty} (\rn) = \big( L_{r_1} (\rn), L_{r_2} (\rn) \big)_{\theta, \infty} \hra \big( \Ca{}^{-\sigma_1} (\rn), \Ca{}^{-\sigma_2} (\rn) \big)_{\theta,\infty} \hra \Ca{}^{-\sigma} (\rn)
\end{\eq}
As far as Lorentz spaces and the above real interpolations are concerned one may consult \cite[Theorem 1.18.6/1, p.\,133, Theorem 2.4.1, 
p.\,182]{T78} (and their proofs). In particular,
\begin{\eq}   \label{3.61}
h_\sigma \in L_{\frac{n}{\sigma}, \infty} (\rn) \hra \Ca{}^{-\sigma} (\rn),
\end{\eq}
justifies to deal with $h_\sigma$ in the context of tempered homogeneous spaces.
\\[0.1cm]
{\em Step 2.} We prove part (i). Let $\vp^j (\xi)= \vp^0 (2^{-j} \xi)$, $j\in \ganz$, be as
in \eqref{2.32}. Then one has by \eqref{3.50}
\begin{\eq}  \label{3.62}
\vp^j (\xi) \, \wh{h_\sigma} (\xi) = c\, 2^{j(\sigma -n)} \, |2^{-j} \xi |^{\sigma -n} \, \vp^0 (2^{-j} \xi ), \qquad j\in \ganz,
\end{\eq}
\begin{\eq}   \label{3.63}
\big( \vp^j (\xi) \, \wh{h_\sigma} (\xi) \big)^\vee (x) = c' \, 2^{j\sigma} \, \big( \vp^0 \, \wh{h_\sigma} \big)^\vee (2^j x), \qquad j \in \ganz,
\end{\eq}
and 
\begin{\eq}   \label{3.64}
\big\| (\vp^j \wh{h_\sigma} )^\vee \, | L_p (\rn) \big\| = c' \, 2^{j(\sigma - \frac{n}{p})} \| (\vp^0 \wh{h_\sigma})^\vee \, | L_p (\rn) \|, \quad
j\in \ganz.
\end{\eq}
Inserted in \eqref{2.87} (what is now justified by Step 1) one obtains 
\begin{\eq}  \label{3.65}
h_\sigma \in \Bas (\rn) \quad \text{if, and only if,} \quad s= \frac{n}{p} -\sigma, \quad q=\infty.
\end{\eq}
By \eqref{3.5} with $v=p <\infty$ one has $h_\sigma \not\in \Fa{}^{\frac{n}{p} -\sigma}_{p,q} (\rn)$. Part (i) follows now from our comments after
\eqref{3.52}.
\\[0.1cm]
{\em Step 3.} We prove part (ii). In the starting terms in \eqref{2.7}, \eqref{2.9} one can replace $\vp_0$ according to \eqref{2.3} with $\vp_0 (x)
= \vp_0 (-x) \ge 0$, $x\in \rn$, by the convolution $\vp = \vp_0 * \vp_0$. Then $\wh{\vp} (x) = c \, \wh{\vp_0}^2 (x) \ge 0$ for some $c>0$. One has
\begin{\eq}   \label{3.66}
\big( \vp \wh{h_\sigma} \big)^\vee (x) = c' \int_{\rn} \vp^\vee (y) \frac{1}{|x-y|^\sigma} \, \di y \ge c'' \, |x|^{-\sigma}, \qquad |x| \ge 1,
\end{\eq}
$c'>0, c''>0$. This shows that $(\vp  \wh{h_\sigma} )^\vee \in L_p (\rn)$ requires $\sigma p >n$ (in particular $p>1$). By \eqref{3.64} with
$j\in \nat$ it follows that $h_\sigma \in \Bs (\rn)$ requires
\begin{\eq}   \label{3.67}
\text{either} \quad s< \frac{n}{p} - \sigma \quad \text{or} \quad s= \frac{n}{p} - \sigma,\ q=\infty, \quad \text{in addition to} \quad \frac{n}{p} < \sigma.
\end{\eq}
From \eqref{3.6} with $v=p<\infty$ (and what already know) follows
$h_\sigma \not\in F^{\frac{n}{p} - \sigma}_{p,q} (\rn)$, $0<p<\infty$, $0<q \le \infty$. On
the other hand using $n \big( \frac{1}{p} -1 \big) < \frac{n}{p} -\sigma <0$ one has by part (i) and the embedding \eqref{2.34}
\begin{\eq}   \label{3.687}
h_\sigma \in B^{\frac{n}{p} - \sigma}_{p,\infty} (\rn) \qquad 0<p \le \infty, \quad \frac{n}{p} < \sigma <n.
\end{\eq}
This justifies the upper line in \eqref{3.56}. The lower line in \eqref{3.56}  follows from \eqref{3.67} and elementary embeddings of inhomogeneous
spaces.
\end{proof}

\begin{remark}   \label{R3.10}
By the above theorem and \eqref{2.88}
\begin{\eq}   \label{3.69}
h_\sigma \in \Ba{}^{\frac{n}{p} - \sigma}_{p,\infty} (\rn) \hra B^{\frac{n}{p} - \sigma}_{p,\infty} (\rn), \qquad \frac{n}{p} <\sigma <n,
\end{\eq}
are the only cases where $h_\sigma$ belongs both to a tempered homogeneous space and its inhomogeneous counterpart.
\end{remark}

So far we excluded the spaces $\Fa{}^s_{\infty,q} (\rn)$ in the above theorem. We rely on the spaces of negative smoothness as introduced in Definition
\ref{D2.2}. The homogeneity \eqref{2.35} applied to \eqref{3.52} shows that $h_\sigma \in \Aa{}^s_{\infty,q} (\rn)$ requires $s= -\sigma$, where again
$0<\sigma <n$. It comes out that $\Ca{}^{-\sigma} (\rn) = \Ba{}^{-\sigma}_{\infty, \infty} (\rn) = \Fa{}^{-\sigma}_{\infty, \infty} (\rn)$ is the only
space with this property. We formulate a corresponding assertion and outline a proof.

\begin{corollary}   \label{C3.11}
Let $n \in \nat$ and $0<\sigma <n$. Let $0<q\le \infty$. Then
\begin{\eq}   \label{3.70}
h_\sigma \in \Aa{}^{-\sigma}_{\infty, q} (\rn) \qquad \text{if, and only if,} \quad q=\infty.
\end{\eq}
\end{corollary}
\begin{proof}
Let $\delta$ be the $\delta$-distribution. By \eqref{2.12} and \eqref{2.24} one has
\begin{\eq}   \label{3.71}
W_t \delta (x) = c \, t^{-n/2} \, e^{-\frac{|x|^2}{4t}}, \qquad t>0, \quad x\in \rn,
\end{\eq}
$c>0$, and
\begin{\eq}   \label{3.72}
\delta \in \Ba{}^{-n}_{\infty, q} (\rn) \qquad \text{if, and only if,} \quad q=\infty.
\end{\eq}
Hence $\delta \in \Ca{}^{-n} (\rn)$. Inserting \eqref{3.71} in \eqref{2.26} with $q<\infty$ one obtains for $t>0$ and some $c'>0$
\begin{\eq}   \label{3.73}
\big\| \delta \, | \Fa{}^{-n}_{\infty, q} (\rn) \big\|^q \ge c' \, t^{-n/2} \int^t_0 \int_{|y| \le \sqrt{t}} \tau^{nq/2} \, \tau^{-nq/2} \, \di y \,
\frac{\di \tau}{\tau} =\infty.
\end{\eq}
Hence
\begin{\eq}  \label{3.74}
\delta \not\in \Fa{}^{-n}_{\infty,q} (\rn), \qquad 0<q<\infty.
\end{\eq}
We extend now the lift $\dot{I}_\sigma$ according to \eqref{3.16} to the usual homogeneous spaces as described in \cite[Proposition 2.18, p.\,23]{T15}.
In particular,
\begin{\eq}  \label{3.75}
\dot{I}_\sigma \, \dot{F}^s_{1,q} (\rn) = \dot{F}^{s-\sigma}_{1,q} (\rn).
\end{\eq}
Recall the duality
\begin{\eq}  \label{3.76}
\dot{F}^s_{1,q} (\rn)' = \dot{F}^{-s}_{\infty, q'} (\rn), \quad s\in \real, \quad 1<q<\infty, \quad \frac{1}{q} + \frac{1}{q'} = 1,
\end{\eq}
covered by \cite[(5.2), p.\,70]{FrJ90}. This gives the possibility to shift the lifting \eqref{3.75} by duality to
\begin{\eq}   \label{3.77}
\dot{I}_\sigma \, \Fa{}^s_{\infty,q} (\rn) = \Fa{}^{s-\sigma}_{\infty, q} (\rn), \qquad s\in \real, \quad 1<q<\infty.
\end{\eq}
Then one has by \eqref{3.50} and $\wh{\delta} =c \not= 0$,
\begin{\eq}  \label{3.78}
h_{n-\sigma} (x) = c \, h^\vee_\sigma (x) = c' \, \big( h_\sigma \, \wh{\delta} \big)^\vee (x) = c' \, \dot{I}_{-\sigma} \delta (x).
\end{\eq}
By \eqref{3.74} one obtains
\begin{\eq}  \label{3.79}
h_{n-\sigma} \not\in \Fa{}^{\sigma -n}_{\infty,q} (\rn), \qquad 1<q<\infty
\end{\eq}
what can be extended by  the embedding \eqref{3.1} to $0<q<\infty$. Then \eqref{3.70} follows from \eqref{3.54} and \eqref{3.79}.
\end{proof}

\begin{remark}   \label{R3.12}
The above sketchy proof is not really satisfactory. We switched from the tempered homogeneous spaces $\Aas (\rn)$ considered in $\big( S\rn), S'(\rn)
\big)$ to their homogeneous counterparts in the framework of $\big( \dot{S} (\rn), \dot{S}' (\rn) \big)$. It would be desirable to prove the above
corollary in the context of tempered homogeneous spaces.
\end{remark}

\begin{remark}   \label{R3.13}
By Definition \ref{D2.2}(iii) and Theorem \ref{T2.6}
\begin{\eq}   \label{3.80}
\BMO^s (\rn) = \Fa{}^s_{\infty, 2} (\rn), \qquad s<0,
\end{\eq}
are tempered homogeneous spaces. Further details may be found in \cite[Remark 3.4, p.\,51 and (3.107), p.\,64]{T15}. One has by \eqref{3.70}
\begin{\eq}  \label{3.81}
|x|^{-\sigma} = h_\sigma (x) \not\in \BMO^{-\sigma} (\rn), \qquad 0<\sigma <n.
\end{\eq}
In particular,
\begin{\eq}   \label{3.82}
|x|^{-1} \not\in \BMO^{-1} (\rn), \qquad \text{but} \quad |x|^{-1} \in \Ca{}^{-1} (\rn),
\end{\eq}
$n \ge 2$. The space $\BMO^{-1} (\rn)$ plays a central role in the recent theory of the Navier-Stokes equations. But it is not so clear whether the
observation \eqref{3.82} is of any interest in this context.
\end{remark}

\subsection{Examples II: Mixed Riesz kernels}    \label{S3.6}
We replace the Riesz kernels $h_\sigma$ according to \eqref{3.49} by
\begin{\eq}   \label{3.83}
h_{\sigma,\vk} (x) = \frac{1}{|x|^\sigma} + \frac{1}{|x|^{\vk}}, \qquad 0< \sigma < \vk <n, \quad x \in \rn,
\end{\eq}
asking for a counterpart of Theorem \ref{T3.9}.

\begin{theorem}   \label{T3.14}
Let $n\in \nat$ and $0<\sigma <\vk <n$. 
\\[0.1cm]
{\upshape (i)} Then
\begin{\eq}   \label{3.84}
h_{\sigma, \vk} \not\in \Aas (\rn)
\end{\eq}
for all spaces $\Aas (\rn)$ as introduced in Definition \ref{D2.8}.
\\[0.1cm]
{\upshape (ii)} Let $\As (\rn)$ with
\begin{\eq} \label{3.85}
0<p \le \infty, \qquad 0 <q\le \infty, \qquad s\in \real,
\end{\eq}
$(p<\infty$ for $F$-spaces$)$ be the inhomogeneous spaces as introduced in \eqref{2.6}--\eqref{2.9}. 
Then $h_{\sigma, \vk} \in \As (\rn)$ if, and only if, $n/p <\sigma <\vk <n$ and
\begin{\eq}   \label{3.86}
\begin{cases}
either &\As (\rn) = B^{\frac{n}{p} - \vk}_{p, \infty} (\rn), \\
or     &0<q\le \infty, \ s< \frac{n}{p} - \vk.
\end{cases}
\end{\eq}
\end{theorem}

\begin{proof}
{\em Step 1.} We prove part (i). By \eqref{3.50} one has
\begin{\eq}   \label{3.87}
\wh{h_{\sigma, \vk}} (\xi) = \frac{c_\sigma}{|\xi |^{n-\sigma}} + \frac{c_{\vk}}{|\xi|^{n - \vk}}, \qquad 0 \not= \xi \in \rn.
\end{\eq}
Assuming $h_{\sigma, \vk} \in \Bas (\rn)$ then it follows from Theorem \ref{T2.10} and \eqref{3.64} that
\begin{\eq}   \label{3.88}
\| h_{\sigma, \vk} \, | \Bas (\rn) \| \sim \Big( \sum^\infty_{j=1} 2^{jq (s+\vk - \frac{n}{p})} \Big)^{1/q} + \Big(\sum_{j=-\infty}^0 2^{jq (s +\sigma
- \frac{n}{p})} \Big)^{1/q}
\end{\eq}
must be finite, that is
\begin{\eq}   \label{3.89}
s - \frac{n}{p} \le - \vk < - \sigma \le s - \frac{n}{p}.
\end{\eq}
This is not possible. As for $\Fas (\rn)$ one can rely on $\Fas (\rn) \hra \Ba{}^s_{p,\infty} (\rn)$.
\\[0.1cm]
{\em Step 2.} We prove part (ii).  By \eqref{3.66} we have for the starting terms in \eqref{2.7}, \eqref{2.9}
\begin{\eq}   \label{3.90}
\big( \vp \,\wh{h_{\sigma, \vk}} \big)^\vee (x) \ge c |x|^{-\sigma}, \qquad |x| \ge 1,
\end{\eq}
for some $c>0$. The converse follows from
\begin{\eq}   \label{3.91}
\int_{|x-y| \le |x|/2} \frac{\wh{\vp} (y) }{|x-y|^\sigma} \, \di y + \int_{|x-y| \ge |x|/2}  \frac{\wh{\vp} (y) }{|x-y|^\sigma} \, \di y \le
c\, |x|^{-N -\sigma +n} + c\, |x|^{-\sigma}
\end{\eq}
where $N$ can be chosen arbitrarily large, in particular $N=n$. Hence
\begin{\eq}   \label{3.92}
\big( \vp \, \wh{h_{\sigma, \vk}} \big)^\vee (x) \sim |x|^{-\sigma}, \qquad |x| \ge 1.
\end{\eq}
Then one has $\big(\vp \wh{h_{\sigma, \vk}} \big)^\vee \in L_p (\rn)$ if, and only if,
$\sigma >n/p$. The  upper line in \eqref{3.86} follows from \eqref{3.6} with $v=p$ and the first term on the right-hand side of \eqref{3.88} based on
\eqref{2.7}. The lower line in \eqref{3.86} is again a matter of elementary embedding.
\end{proof}

\subsection{Wavelets}   \label{S7}
In Section \ref{S3.8} we deal with cocompact embeddings and nonlinear approximation. We rely on wavelet expansions for $\Aas (\rn)$ in the distinguished
strip
\begin{\eq}   \label{3.93}
0<p,q \le \infty, \qquad n \big( \frac{1}{p} - 1 \big) <s< \frac{n}{p} \, ,
\end{\eq}
in the framework of the dual pairing $\big( S(\rn), S' (\rn) \big)$. For this purpose we transfer corresponding characterizations of the inhomogeneous
spaces $\As (\rn)$ to the above tempered homogeneous spaces $\Aas (\rn)$ excluding limiting cases. Under the restriction \eqref{3.93} there are no
topological  problems and one can carry over the final version for wavelet expansions of  $\As (\rn)$ according to \cite[Theorem 1.18, pp.\,10/11]{T10}.
There one finds also some further references. Corresponding wavelet expansions for homogeneous spaces $\dot{A}^s_{p,q} (\rn)$ are discussed in 
\cite[Section 3.5.1, pp.\,82--88]{T14} based again on some references. We adapt what follows to our later needs.

Let again
\begin{\eq}   \label{3.94}
\Ca{}^r (\rn) = \Ba{}^r_{\infty, \infty} (\rn), \qquad -n<r<0,
\end{\eq}
be the {\em anchor space} of the spaces $\Aas (\rn)$ with $s- \frac{n}{p} =r$, normed by \eqref{3.0b} = \eqref{2.29}. One has by Definition \ref{D2.8}
and Theorem \ref{T3.1}
\begin{\eq}   \label{3.95}
\Aas (\rn) \hra \Ca{}^r (\rn), \qquad s- \frac{n}{p} =r,
\end{\eq}
for all $s,p,q$ with \eqref{3.93}. We take over some material from \cite[Section 1.1.4]{T10}.  Let
\begin{\eq}   \label{3.96}
\psi_F \in C^{u} (\real), \quad \psi_M \in C^{u} (\real), \qquad u \in \nat,
\end{\eq}
be real compactly supported Daubechies wavelets with
\begin{\eq}    \label{3.97}
\int_{\real} \psi_M (x) \, x^v \, \di x =0 \qquad \text{for all $v\in \no$ with $v<u$.}
\end{\eq}
We extend these wavelets from $\real$ to $\rn$ by the usual multiresolution procedure. Let
\begin{\eq}   \label{3.98}
G = (G_1, \dots, G_n ) \in G^n = \{ F,M \}^{n*}
\end{\eq}
where  $G_l$ is either $F$ or $M$ and where $*$ indicates that at least one of the $2^n -1$ components of $G$ must be an $M$. Let
\begin{\eq}   \label{3.99}
\Psi^{j,r}_{G,m} (x) = 2^{-jr} \prod^n_{l=1} \psi_{G_l} \big( 2^j x_l - m_l), \qquad j\in \ganz, \quad G\in G^n, \quad m\in \zn,
\end{\eq}
with $r$ as in \eqref{3.94}. These are normalized building blocks in $\Ca{}^r (\rn)$ if $u > |r|$ in \eqref{3.96}. This applies also to the spaces
$\Aas (\rn)$ in \eqref{3.95} if $u> |s|$ (and $u> n (\frac{1}{q} -1) -s$ in case of $F$-spaces). As a consequence we adapt the sequence spaces 
$b^s_{p,q}$ and $f^s_{p,q}$ in \cite[p.\,9]{T10} to the above situation as follows. Let again $Q_{j,m} = 2^{-j} m + 2^{-j} (0,1)^n$ with $j\in \ganz$
and $m \in \zn$ be  as in Section \ref{S2.1}. Let $\chi_{j,m}$ be the characteristic function of $Q_{j,m}$ and $\chi^{(p)}_{j,m} (x) = 2^{j \frac{n}{p}}
\chi_{j,m} (x)$, $0<p\le \infty$, its $p$-normalization. Let $0<q\le \infty$. Then $\ba$ is the collection of all sequences
\begin{\eq}   \label{3.100}
\lambda = \big\{ \lambda^{j,G}_m \in \comp: \ j\in \ganz, \ G \in G^n, \ m \in \zn \big\}
\end{\eq}
such that
\begin{\eq}   \label{3.101}
\| \lambda \, | \ba \| = \bigg( \sum_{j\in \ganz} \sum_{G\in G^n} \Big( \sum_{m \in \zn} |\lambda^{j,G}_m |^p \Big)^{q/p} \bigg)^{1/q} <\infty
\end{\eq}
and $\fa$ is the collection of all sequences $\lambda$ in \eqref{3.100} such that
\begin{\eq}   \label{3.102}
\| \lambda \, | \fa \| = \Big\| \Big( \sum_{j,G,m} | \lambda^{j,G}_m \chi^{(p)}_{j,m} (\cdot)|^q \Big)^{1/q} | L_p (\rn) \Big\| <\infty
\end{\eq}
with the usual modification if $p=\infty$ and/or $q=\infty$. This is the adapted homogeneous counterpart of \cite[Definition 1.17, p.\,9]{T10}. In 
\cite[Theorem 1.18, pp.\,10/11]{T10} we characterized the elements of the inhomogeneous spaces $\As (\rn)$ in the framework of $\big( S(\rn), S'(\rn)
\big)$ by their expansions in terms of inhomogeneous wavelets. This point of view cannot be extended to the above tempered homogeneous spaces. As
already indicated in front of Theorem \ref{T2.10} and in \cite[Section 1.3, pp.\,5/6]{T15} we ask for {\em domestic} characterizations and {\em 
domestic quasi-norms} within an already fixed space $\Aas (\rn)$. As in the case of the inhomogeneous spaces we justify first that the dual pairing
\begin{\eq}   \label{3.103}
(f,\psi), \qquad f \in \Aas (\rn), \quad \psi \in C^{u} (\rn), \quad \text{$\supp \psi \ $ compact},
\end{\eq}
with $p,q,s$ as in \eqref{3.93} and $|s|<u \in \nat$ makes sense. If, in addition, $s<0$ then \eqref{3.103} follows from \eqref{2.88} and a well-known
corresponding assertion for related inhomogeneous spaces. The extension of this assertion to all admitted spaces with $p,q,s$ as in \eqref{3.93} is
now a matter of the embedding
\begin{\eq}   \label{3.103a}
\Aas (\rn) \hra \Ba{}^\sigma_{\tilde{p}, \infty} (\rn), \qquad s - \frac{n}{p} = \sigma - \frac{n}{\tilde{p}} =r
\end{\eq}
with $s>\sigma$ and $\sigma <0$. Recall that $\Psi^{j,r}_{G,m}$ is given by \eqref{3.99} based on \eqref{3.96}, \eqref{3.97}.

\begin{theorem}   \label{T3.15}
Let $n\in \nat$ and $-n<r<0$. 
\\[0.1cm]
{\upshape (i)} Let $0<p\le \infty$, $0<q\le \infty$ and $s - \frac{n}{p} =r$. Let $|s| <u \in \nat$. Let
\begin{\eq}   \label{3.104}
\lambda^{j,G}_m = \lambda^{j,G}_m (f) = 2^{j(n+2r)} \big( f, \Psi^{j,r}_{G,m} \big),
\end{\eq}
$f \in \Bas (\rn)$, with $j\in \ganz$, $G\in G^n$, $m\in \zn$. Then $f$ can be represented as
\begin{\eq}   \label{3.105}
f = \sum_{j,G,m} \lambda^{j,G}_m (f) \, \Psi^{j,r}_{G,m}
\end{\eq}
unconditional convergence being in $S'(\rn)$. Furthermore $\lambda = \lambda (f) \in \ba$ and $\| \lambda (f) \, | \ba \|$ is an equivalent {\em
domestic} quasi-norm in $\Bas (\rn)$. If $p<\infty$, $q< \infty$ then $\{ \Psi^{j,r}_{G,m} \}$ is a  normalized unconditional basis in $\Bas (\rn)$.
\\[0.1cm]
{\upshape (ii)} Let $0<p<\infty$, $0<q\le \infty$ and $s- \frac{n}{p} =r$. Let $\max \big( s, n (\frac{1}{q} -1) -s \big) <u \in \nat$. Then $f\in 
\Fas (\rn)$ can be represented by \eqref{3.105} with $\lambda^{j,G}_m (f)$ as in \eqref{3.104}, unconditional convergence being in $S'(\rn)$. 
Furthermore, $\lambda = \lambda (f) \in \fa$ and $\| \lambda (f) \, | \fa \|$ is an equivalent {\em domestic} quasi-norm in $\Fas (\rn)$. If $q<\infty$
then $\{ \Psi^{j,r}_{G,m} \}$ is a normalized unconditional basis.
\end{theorem}

\begin{remark}   \label{R3.16}
A detailed proof of the inhomogeneous counterpart maybe found in \cite[Section 3.1.3, Theorem 3.5, pp.\,153--156]{T06} and \cite[Theorem 1.20,
pp.\,15--17]{T08}, repeated in \cite[Theorem 1.18, pp.\,10/11]{T10}. This can be transferred to the above tempered homogeneous spaces based on related
homogeneous counterparts as discussed in \cite{T15}. We do not go into detail. Otherwise the above theorem is more or less known, at least  on the level
of mathematical folklore, in the framework of $\big( \dot{S} (\rn), \dot{S}'  (\rn) \big)$ and then in the full range of $0 <p,q \le \infty$ and $s\in
\real$. We used already wavelet expansions of homogeneous spaces $\dot{A}^s_{p,q} (\rn)$ in \cite[p.\,84/85]{T14}, \cite[p.\,62]{T15} with a reference 
to \cite{LSUYY12}.
\end{remark}

\subsection{Cocompactness and greedy approximation}   \label{S3.8}
First we complement Theorem \ref{T3.1}. Let
\begin{\eq}   \label{3.106}
0<p_0 <p_1 <\infty, \qquad -n <s_0 - \frac{n}{p_0} = s_1 - \frac{n}{p_1} =r <0.
\end{\eq}
Then
\begin{\eq}   \label{3.107}
\Fa{}^{s_0}_{p_0, q_0} (\rn) \hra \Fa{}^{s_1}_{p_1, q_1} (\rn) \hra \Ca{}^r (\rn)
\end{\eq}
where again $\Ca{}^r (\rn)$ is the anchor space \eqref{3.94}, \eqref{3.95} and $0<q_0, q_1 \le \infty$. The proof of \eqref{3.107} can be based on the
inhomogeneous counterpart according to \cite[Theorem 2.7.1, p.\,129]{T83} and the same homogeneity argument as in Step 3 of the proof of Theorem 
\ref{T3.1}. Using the wavelet representation according to Theorem \ref{T3.15} one finds a bounded sequence $\{ f_k \}^\infty_{k=1} \subset 
\Fa{}^{s_0}_{p_0, q_0} (\rn)$ such that $\supp f_k \subset Q = (0,1)^n$ and
\begin{\eq}   \label{3.108}
1 \le \| f_{k_1} - f_{k_2} \, | \Ca{}^r (\rn) \| \le \| f_{k_1} - f_{k_2} \, | \Fa{}^{s_1}_{p_1, q_1} (\rn) \|, \qquad k_1 \not= k_2.
\end{\eq}
This shows that neither the embeddings in \eqref{3.107} nor their restrictions to bounded domains are compact. But they are {\em cocompact} with
respect to the isomorphic maps generated by the wavelet system
\begin{\eq}   \label{3.109}
\Psi = \big\{ \Psi^{j,r}_{G,m}: \ j\in \ganz, \ G\in G^n, \ m\in \zn \}
\end{\eq}
according to \eqref{3.99} in the following sense: Let 
\begin{\eq}   \label{3.110}
\big( I^{j,r}_m f \big)(x) = 2^{jr} \, f \big( 2^{-j} (x+m) \big), \qquad j\in \ganz \quad m \in \zn,
\end{\eq}
in the distributional understanding. With $\Aas (\rn)$ as in \eqref{3.95} one has by \eqref{1.3}
\begin{\eq}   \label{3.111}
\| I^{j,r}_m f \, | \Aas (\rn) \| = 2^{jr} \, 2^{-j (s- \frac{n}{p})} \| f \, | \Aas (\rn) \| = \| f \, | \Aas (\rn) \|.
\end{\eq}
Let $\psi_G (x)= \prod^n_{l=1} \psi_{G_l} (x_l)$ with $G\in G^n$ be a basic function in \eqref{3.99}. By \eqref{3.104} one has
\begin{\eq}   \label{3.106a}
\begin{aligned}
\big( I^{j,r}_m f, \psi_G \big) &= 2^{jr} \int_{\rn} f \big( 2^{-j} (x+m) \big) \, \psi_G (x) \, \di x \\
&= 2^{j(r+n)} \int_{\rn} f(y) \, \psi_G \big( 2^j y -m \big) \, \di y \\
&= \lambda^{j,G}_m (f)
\end{aligned}
\end{\eq}
(distributional interpretation). The (first or second) embedding in \eqref{3.107} is called {\em cocompact} (with respect to the isomorphisms 
$I^{j,r}_m$) if every bounded sequence $\{f_k \}^\infty_{k=1}$ in the source space such that
\begin{\eq}   \label{3.107a}
\big( I^{j_k,r}_{m_k} f_k, \vp \big) \to 0 \qquad \text{for any choice of $j_k$, $m_k$}
\end{\eq}
and any $\vp \in S(\rn)$ converges to zero in the target space (weak* convergence). First we justify that \eqref{3.107a} remains valid for compactly
supported functions $\psi \in C^{u} (\rn)$ with $|r| <u \in \nat$: By the same arguments as in \cite[pp.\,63/64]{T15} one has the duality
\begin{\eq}   \label{3.108a}
\Ba{}^{|r|}_{1,1} (\rn)' = \Ca{}^r (\rn)
\end{\eq}
as in the inhomogeneous counterpart, \cite[Theorem 2.11.2, p.\,178]{T83}. By \cite[Proposition 3.52, p.\,107]{T15} the spaces $B^{|r|}_{1,1} (\rn)$
and $\Ba{}^{|r|}_{1,1} (\rn)$ coincide locally. Recall that $S(\rn)$ is dense in $B^{|r|}_{1,1} (\rn)$. Then \eqref{3.107a} with $\psi$ in place of
$\vp$ follows from
\begin{\eq}   \label{3.109a}
\big| \big( I^{j_k,r}_{m_k} f_k, \psi \big) \big| \le \big| \big( I^{j_k,r}_{m_k} f_k, \vp \big) \big| + c \, \| \vp - \psi \, | B^{|r|}_{1,1} (\rn) \|
\end{\eq}
and suitable approximations. Applying this observation to \eqref{3.106a} one has by \eqref{3.107a} and Theorem \ref{T3.15}
\begin{\eq}   \label{3.110a}
\| f_k \, | \Ca{}^r (\rn) \| \sim \sup_{j,G,m} | \lambda^{j,G}_m (f_k) | \to 0 \qquad \text{if} \quad k \to \infty.
\end{\eq}
Hence the embedding
\begin{\eq}   \label{3.111a}
\Fa{}^{s_0}_{p_0, q_0} (\rn) \hra \Ca{}^r (\rn)
\end{\eq}
is cocompact. We extend this assertion to the first embedding in \eqref{3.107}. Let  $-\infty <t_1 <t_2 <\infty$, $0<q\le \infty$, $0<\theta <1$ and
$t = (1-\theta) t_1 + t \theta_2$. Then for $\{ a_j \}^\infty_{j=-\infty} \subset \comp$ one has
\begin{\eq} \label{3.112}
\Big( \sum^\infty_{j=-\infty} 2^{jtq} |a_j |^q \Big)^{1/q} \le C \, \Big( \sup_{j\in \ganz} 2^{j t_1} |a_j | \Big)^{1-\theta} \Big( \sup_{j\in \ganz}
2^{j t_2} |a_j | \Big)^\theta
\end{\eq}
(usual modification if $q=\infty$), \cite[Lemma 3.7, p.\,394]{BrM01}. The proof uses the same splitting technique as in \cite[pp.\,129/130]{T83}, based
on \cite{Jaw77}, resulting in the first embedding in \eqref{3.107}. Let
\begin{\eq}   \label{3.113}
0<p_1 <p_2 <\infty, \qquad -n < s_1 - \frac{n}{p_1} = s_2 - \frac{n}{p_2} =r <0,
\end{\eq}
$0<\theta <1$,
\begin{\eq}   \label{3.114}
\frac{1}{p} = \frac{1-\theta}{p_1} + \frac{\theta}{p_2}, \qquad s = (1-\theta)s_1 + \theta s_2
\end{\eq}
and $0<q, q_1, q_2 \le \infty$. Then   
\begin{\eq}   \label{3.115}
\| f \, |\Fas (\rn) \| \le c \, \| f \, | \Fa{}^{s_1}_{p_1, q_1} (\rn) \|^{1-\theta} \| f \, | \Fa{}^{s_2}_{p_2, q_2} (\rn) \|^\theta,
\end{\eq}
$f\in \Fa{}^{s_1}_{p_1, q_1} (\rn)$, follows from \eqref{3.112} (with $t_1 = \frac{n}{p_1}$, $t_2 = \frac{n}{p_2}$, $t = \frac{n}{p}$) inserted in the
wavelet representation according to Theorem \ref{T3.15}(ii), based on \eqref{3.102}, H\"{o}lder's inequality and $\Fas (\rn) \hra \Fa{}^s_{p,\infty}
(\rn)$. This assertion can be complemented by
\begin{\eq}  \label{3.116}
\| f \, | \Fas (\rn) \| \le c \, \| f \, | \Fa{}^{s_0}_{p_0, q_0} (\rn) \|^{1-\theta} \| f \, | \Ca{}^r (\rn) \|^\theta,
\end{\eq}
$f \in \Fa{}^{s_0}_{p_0, q_0} (\rn)$, with
\begin{\eq}  \label{3.117}
0<p_0 <\infty, \qquad -n <s_0 - \frac{n}{p_0} = r <0,
\end{\eq}
$0< \theta <1$,
\begin{\eq}   \label{3.118}
\frac{1}{p} = \frac{1-\theta}{p_0}, \qquad s = (1-\theta)s_0 + \theta r
\end{\eq}
and $0<q, q_0 \le \infty$. Both \eqref{3.115} and \eqref{3.116} are Gagliardo--Nirenberg refinements of \eqref{3.107}. They go back to \cite[Lemma 3.1,
p.\,393]{BrM01} (with a reference to Oru) and have also been used in \cite{HaS11} by the same wavelet arguments as above, \cite[Lemma 5,
p.\,938]{HaS11}. The cocompactness of the first embedding in \eqref{3.107} follows now from the cocompactness of the second embedding and \eqref{3.116}.

Cocompactness may be considered as a substitute if embeddings  between spaces are continuous but not compact. A discussion from the point of view
of function spaces may be found in the recent paper \cite{Tin16} and the literature within (one of the proofs presented there is quite similar as the
above arguments based on wavelet representations). Cocompactness is closely related to so-called profile decompositions of critical embeddings between
function  spaces which are continuous but not compact playing a role in the recent theory of non-linear PDE's, including Navier-Stokes equations.
Related details may be found in \cite{Jaf99, BCK11}. We do not discuss these topics here although they might be of some use in connection with our
approach to Navier-Stokes and Keller-Segel equations in \cite{T13, T14, T16}. 

We are interested in so-called best $k$-term approximations and greedy bases which may be regarded as quantitative aspects of cocompactness. What
follows can also be considered as a comment to \cite{Kyr01} and, in particular, \cite{HaS11} restricting ourselves exclusively to the embeddings
\eqref{3.106}, \eqref{3.107}, that is
\begin{\eq}   \label{3.119}
\id: \quad F_0 (\rn) = \Fa{}^{s_0}_{p_0, q_0} (\rn) \hra F_1 (\rn)
\end{\eq}
where $F_1 (\rn)$ is either $\Fa{}^{s_1}_{p_1, q_1} (\rn)$ or $\Ca{}^r (\rn)$. Let $\Psi$ be the wavelet system \eqref{3.109} (fixed once and for all
in the sequel). Let $k\in \nat$. Then
\begin{\eq}   \label{3.120}
\sigma_k (f,F_1 (\rn) ) = \inf \Big\{ \big\| f - \sum_{(j,G,m) \in K} \mu^{j,G}_m \Psi^{j,r}_{G,m} \, | F_1 (\rn) \big\|: \ \text{card }K=k, \quad
\mu^{j,G}_m \in \comp \Big\}
\end{\eq}
is the {\em best $k$-term approximation} of $f \in \Fa{}^{s_1}_{p_1, q_1} (\rn)$ and
\begin{\eq}   \label{3.121}
\sigma_k (\id) = \sup \big\{ \sigma_k \big( f, F_1 (\rn) \big): \ f \in F_0 (\rn), \ \| f \, | F_0 (\rn) \| \le 1 \big\}.
\end{\eq}
Let the coefficients  of the wavelet expansion of $f \in F_1 (\rn)$ in \eqref{3.105} be ordered by magnitude,
\begin{\eq}   \label{3.122}
\big| \lambda^{j_1, G^{(1)}}_{m^{(1)}} (f)\big| \ge \big| \lambda^{j_2, G^{(2)}}_{m^{(2)}} (f) \big| \ge \cdots .
\end{\eq}
Then
\begin{\eq}   \label{3.123}
g_k (f) = \sum^k_{l=1} \lambda^{j_l, G^{(l)}}_{m^{(l)}} (f) \, \Psi^{j_l,r}_{G^{(l)}, m^{(l)}}, \qquad k \in \nat,
\end{\eq}
is called the {\em greedy algorithm} in $F_1 (\rn)$. The system $\Psi$ is said to be {\em greedy} in $F_1 (\rn)$ if there is a constant $C$ ($\ge 1$)
such that for all $f\in F_1 (\rn)$ and all $k\in \nat$
\begin{\eq}   \label{3.124}
\| f - g_k (f) \, | F_1 (\rn) \| \le C \, \sigma_k \big(f, F_1 (\rn) \big).
\end{\eq}
The converse inequality is obvious by definition.

\begin{proposition}    \label{P3.17}
Let $F_1 (\rn) = \Fa{}^{s_1}_{p_1, q_1} (\rn)$ be as in \eqref{3.106}, \eqref{3.107} with $q_1 <\infty$. Then $\Psi$ is a greedy basis in $F_1 (\rn)$.
\end{proposition}

\begin{remark}   \label{R3.18}
We dealt in \cite[Section 6.3, pp.\,210--215]{T08} with greedy bases in several types of spaces \As. The above proposition can be proved by the same
arguments as there. In \cite{T08} one finds also further information and references. In particular it comes out that the wavelet basis $\Psi$ is not greedy in $\Bas (\rn)$ if $p \not= q$.
\end{remark}

It is the main aim of this Section \ref{S3.8} to justify the following assertion.

\begin{theorem}  \label{T3.19}
Let $F_0 (\rn)$ and $F_1 (\rn)$ be as in \eqref{3.106}, \eqref{3.107} and \eqref{3.119}. Then there are two constants $0<c_1 <c_2 <\infty$ such that
\begin{\eq}   \label{3.125}
c_1 \, k^{\frac{1}{p_1} - \frac{1}{p_0}} \le \sigma_k (\id ) \le c_2 \, k^{\frac{1}{p_1} - \frac{1}{p_0}}, \qquad k\in \nat,
\end{\eq}
$($with $p_1 = \infty$ if $F_1 (\rn) = \Ca{}^r (\rn))$.
\end{theorem}

\begin{proof} {\em Step 1.} Let $F_1 (\rn) = \Ca{}^r (\rn)$ and $F_0 (\rn) = \Fas (\rn) = \Fa{}^{s_0}_{p_0, q_0} (\rn)$ as above. Let $f\in \Fas (\rn)$
with $\| f \, | \Fas (\rn) \| \le 1$ be expanded by \eqref{3.105}. We assume that the $k$ largest coefficients according to \eqref{3.122} are larger
than or equal to $\ve_k >0$. Let $\chi^{(p)}_l$ be the $p$-normalized characteristic functions of the related cubes $Q^l = Q_{j_l,m_l}$ in
\eqref{3.102}. Then
\begin{\eq}   \label{3.126}
\ve^p_k \, k = \sum^k_{l=1} \ve^p_k \int_{\rn} \chi^{(p)}_l (x)^p \, \di x \sim \int_{\rn} \Big( \sum^k_{l=1} \ve^q_k \, \chi^{(p)q}_l (x) \Big)^{p/q}
\, \di x \le c,
\end{\eq}
where the equivalence $\sim$ is covered by corresponding arguments in \cite[p.\,213]{T08} and the last estimate comes from \eqref{3.102}. One has 
$\ve_k \le c \, k^{-1/p}$. Then
\begin{\eq}   \label{3.127}
\sigma_k \big( \id: \ \Fas (\rn) \hra \Ca{}^r (\rn) \big) \le c \, k^{-1/p}, \qquad k \in \nat,
\end{\eq}
follows from \eqref{3.120}, \eqref{3.121} and the norm of $\Ca{}^r (\rn)$ as used in \eqref{3.110a}. The converse can be obtained by choosing suitable
normed functions in $\Fas (\rn)$ consisting of $k+1$ terms with $\ve_k = k^{-1/p}$ in the expansion \eqref{3.105}.
\\[0.1cm]
{\em Step 2.} Let now $\Fa{}^{s_0}_{p_0, q_0} (\rn)$ and $\Fa{}^{s_1}_{p_1, q_1} (\rn) = \Fas (\rn)$ be as in \eqref{3.107} and \eqref{3.116}. Then one
has by Step 1 and \eqref{3.118}
\begin{\eq}   \label{3.128}
\sigma_k \big( \id: \ \Fa{}^{s_0}_{p_0, q_0} (\rn) \hra \Fa{}^{s_1}_{p_1,q_1} (\rn) \big) \le c\, k^{-\frac{\theta}{p_0}} = c \, k^{\frac{1}{p_1} - 
\frac{1}{p_0}}, \qquad k \in \nat.
\end{\eq}
The converse is again a matter of special functions consisting of, say, $2k$ terms in \eqref{3.105} with coefficients $k^{-1/p_0}$ and appropriately
located cubes $Q_{j,m}$ in \eqref{3.102}.
\end{proof}

\begin{remark}   \label{R3.20}
As mentioned above best $k$-term approximations both for $B$-spaces and $F$-spaces attracted some attention over the years. We refer the reader to
\cite{Kyr01, HaS11, BCK11} and the literature within. We dealt here exclusively with $F$-spaces where Theorem \ref{T3.19} complements corresponding
assertions in the just-mentioned papers for homogeneous spaces $\dot{F}^s_{p,q} (\rn)$ and inhomogeneous spaces $\Fs (\rn)$ (with the same outcome as in
\eqref{3.125}). As already indicated in Remark \ref{R3.18} Proposition \ref{P3.17} cannot be extended to $\Bas(\rn)$ (or $\dot{B}^s_{p,q} (\rn)$ or
$\Bs (\rn)$) with $p\not=q$. There are counterparts of Theorem \ref{T3.19} for $B$-spaces but one needs additional restrictions for the $q$-parameters
and also modifications of \eqref{3.125} (caused by the observation that there is no $B$-version of \eqref{3.115} for all $q, q_1, q_2$).
\end{remark}

Theorem \ref{T3.19} covers some interesting special cases. This applies in particular to the source spaces
\begin{\eq}  \label{3.129}
F_0 (\rn) = \Fa{}^l_{p_0,2} (\rn) = \Wa{}^l_{p_0} (\rn), \qquad 1<p_0 <\infty, \quad \frac{n}{p_0} >l \in \nat,
\end{\eq}which can be furnished with the equivalent domestic norms
\begin{\eq}    \label{3.130}
\| f \, | \Wa{}^l_{p_0} (\rn) \| = \sum_{|\alpha| =l} \| D^\alpha f \, | L_{p_0} (\rn) \|,
\end{\eq}
\cite[Remark 3.42, p.\,98]{T15} (tempered homogeneous Sobolev  spaces). It is well known that \eqref{3.129}, \eqref{3.130} cannot be extended to 
$p_0 =1$. But related Sobolev spaces attracted a lot of attention. We indicate how these distinguished spaces can be incorporated in our context and 
how Theorem \ref{T3.19} looks like in these cases. According to Theorem \ref{T3.15} the spaces $\Ba{}^s_{1,1} (\rn) = \Fa{}^s_{1,1} (\rn)$ with $0<s<n$
have the equivalent domestic norms \eqref{3.101} with $\bb = \ell_1$. We replace $\ell_1$ in \eqref{3.101} by $\ell_{1,\infty}$, the {\em weak} 
$\ell_1$-space or Marcinkiewicz space (one may consult \cite[Section 1.18.3, pp.\,125--127]{T78}) and denote the outcome by $w\Ba{}^s_{1,1} (\rn)$. 
These spaces can be introduced by weak {\em admissible} modifications of the first term on the right-hand side of \eqref{2.69} having corresponding
{\em domestic} norms as in \eqref{2.87} and, as said, in \eqref{3.101}, \eqref{3.104}.
This will not be done here in detail. Let $n \ge 2$. One has as a by-product
of \cite{CDPX99, CDDD03} for some $c_1 >0$, $c_2 >0$,
\begin{\eq}   \label{3.131}
c_1 \, \| f \, | w\Ba{}^1_{1,1} (\rn) \| \le \sum_{|\alpha|=1} \| D^\alpha f \, | L_1 (\rn) \| \le c_2 \, \| f \, | \Ba{}^1_{1,1} (\rn) \|,
\end{\eq}
$f\in \Ba{}^1_{1,1} (\rn)$. Then $\Wa{}^1_1 (\rn)$ can be introduced as the collection of all $f\in w\Ba{}^1_{1,1} (\rn)$ such that
\begin{\eq}   \label{3.132}
\| f \, | \Wa{}^1_1 (\rn) \| = \sum_{|\alpha|=1} \| D^\alpha f \, | L_1 (\rn) \|
\end{\eq}
is finite. According to \cite[Proposition 3.41, p.\,97]{T15} one has
\begin{\eq}   \label{3.133}
\| f \, | \Aas (\rn) \| \sim \sum_{|\alpha| =m} \| D^\alpha f \, | \Aa{}^{s-m}_{p,q} (\rn) \|
\end{\eq}
within the distinguished strip \eqref{2.67} (equivalent domestic norms). This applies not only to $\Ba{}^l_{1,1} (\rn)$ but also to $w\Ba{}^l_{1,1}
(\rn)$ with $n>l \in \nat$. Then one can extend \eqref{3.131} to
\begin{\eq}   \label{3.134}
c_1 \, \| f \, | w\Ba{}^l_{1,1} (\rn) \| \le \sum_{|\alpha|=l} \| D^\alpha f \, | L_1 (\rn) \| \le c_2 \, \| f \, | \Ba{}^l_{1,1} (\rn) \|,
\end{\eq}
$f\in \Ba{}^l_{1,1} (\rn)$. Afterwards one can introduce $\Wa{}^l_1 (\rn)$ as the collection of all $f\in w\Ba{}^l_{1,1} (\rn)$ such that
\begin{\eq}   \label{3.135}
\| f \, | \Wa{}^l_1 (\rn) \| = \sum_{|\alpha|=l} \| D^\alpha f \, | L_1 (\rn) \|
\end{\eq}
is finite. These observations can be used to extend Theorem \ref{T3.19} to the source spaces $F_0 (\rn) = \Wa{}^l_1 (\rn)$, $n>l \in \nat$. For this
purpose one needs the Gagliardo--Nirenberg inequality
\begin{\eq}   \label{3.136}
\| f\, | \Fa{}^{l- \frac{n}{p'}}_{p,q} (\rn) \| \le c \, \| f \, | \Wa{}^l_1 (\rn) \|^{1/p} \, \| f \, | \Ca{}^{l-n} (\rn) \|^{1/p'}
\end{\eq}
where $1<p<\infty$, $\frac{1}{p} + \frac{1}{p'} =1$, $0<q\le \infty$ and again $n>l \in \nat$. This substitute of \eqref{3.115} may be of interest for
its own sake. The counterpart of \eqref{3.119} is now given by
\begin{\eq}   \label{3.137}
\id: \quad F_0 (\rn) = \Wa{}^l_1 (\rn) \hra F_1 (\rn), \qquad n>l \in \nat,
\end{\eq}
where $F_1 (\rn)$ is either $\Fa{}^{l- \frac{n}{p'}}_{p,q} (\rn)$ with $1<p<\infty$, $\frac{1}{p} + \frac{1}{p'} =1$, $0<q\le \infty$ or $\Ca{}^{l-n}
(\rn)$. Let $\sigma_k (\id)$ be as in \eqref{3.120}, \eqref{3.121}.

\begin{corollary}   \label{C3.21}
Let $n>l \in \nat$, $1<p<\infty$, $\frac{1}{p} + \frac{1}{p'}= 1$, $0<q\le \infty$. Then the embeddings \eqref{3.137} are continuous with \eqref{3.136}.
Furthermore there are two constants $0<c_1 <c_2 <\infty$ such that
\begin{\eq}   \label{3.138}
c_1 \, k^{-1/p'} \le \sigma_k (\id) \le c_2 \, k^{-1/p'}, \qquad k \in \nat,
\end{\eq}
$($with $p'=1$ if $F_1 (\rn) = \Ca{}^{l-n} (\rn))$.
\end{corollary}

\begin{proof} (Outline) {\em Step 1.} By Theorem \ref{T3.19} one has
\begin{\eq}   \label{3.139}
\sigma_k (\id) \sim k^{-1} \qquad \text{if} \quad \id: \ \Ba{}^l_{1,1} (\rn) \hra \Ca{}^{l-n} (\rn).
\end{\eq}
Recall that the supremum over $\ve_k k$ in \eqref{3.126} gives the $\ell_{1,\infty}$-norm in question, \cite[Lemma 1.18.6, p.\,132]{T78}, \cite[Theorem
3.15, p.\,69]{T15}. 
Then the arguments in Step 1 of the proof of Theorem \ref{T3.19} apply also to $w\Ba{}^l_{1,1} (\rn)$ resulting in
\begin{\eq}   \label{3.140}
\sigma_k (\id) \sim k^{-1} \qquad \text{if} \quad \id: \ w\Ba{}^l_{1,1} (\rn) \hra \Ca{}^{l-n} (\rn).
\end{\eq}
Now one has by \eqref{3.134}
\begin{\eq}   \label{3.141}
\sigma_k (\id) \sim k^{-1} \qquad \text{if} \quad \id: \ \Wa{}^l_1 (\rn) \hra \Ca{}^{l-n} (\rn).
\end{\eq}
This proves \eqref{3.138} with $p' =1$ in case of $F_1 (\rn) = \Ca{}^{l-n} (\rn).$
\\[0.1cm]
{\em Step 2.} By Theorem \ref{T3.15}  one has for $\Ba{}^l_{1,1} (\rn)$ the equivalent norm \eqref{3.101} with $\bb = \ell_1$ and for $\Ca{}^{l-n} 
(\rn)$ the equivalent norm \eqref{3.101} with $\bc =\ell_\infty$. Recall the real interpolation
\begin{\eq}   \label{3.142}
\ell_p = (\ell_1, \ell_\infty )_{\theta,p} = (\ell_{1,\infty}, \ell_\infty )_{\theta, p}, \qquad 1<p<\infty, \quad 1- \theta = \frac{1}{p},
\end{\eq}
\cite[Section 1.18.3, pp.\,125--127]{T78} (extended to $\ell_{1,\infty}$). But $\ell_p = \bd$ is an equivalent norm for $\Ba{}^{l- \frac{n}{p'}}_{p,p}
(\rn)$. This proves \eqref{3.137} with $q=p$ and by \eqref{3.134} also
\begin{\eq}   \label{3.143}
\| f \, | \Ba{}^{l - \frac{n}{p'}}_{p,p} (\rn)\| \le c \, \| f \, | \Wa{}^l_1 (\rn) \|^{1/p} \, \| f \, | \Ca{}^{l-n} (\rn) \|^{1/p'}.
\end{\eq}
By Step 1 and the counterpart of \eqref{3.128} one has
\begin{\eq}   \label{3.144}
\sigma_k \big( \id: \ \Wa{}^l_1 (\rn) \hra \Ba{}^{l- \frac{n}{p'}}_{p,p} (\rn) \big) \le c\, k^{-1/p'}, \qquad k \in \nat.
\end{\eq}
Recall that $\Wa{}^l_1 (\rn)$ is sandwiched between $\Ba{}^l_{1,1} (\rn)$ and $w\Ba{}^l_{1,1} (\rn)$,
\begin{\eq}   \label{3.145}
\Ba{}^l_{1,1} (\rn) \hra \Wa{}^l_1 (\rn) \hra w\Ba{}^l_{1,1} (\rn).
\end{\eq}
Then the converse of \eqref{3.144} follows from Theorem \ref{T3.19} applied to $\Ba{}^l_{1,1} (\rn)$. This proves \eqref{3.137} with $F_1 (\rn) =
\Fa{}^{l- \frac{n}{p'}}_{p,p} (\rn) = \Ba{}^{l- \frac{n}{p'}}_{p,p} (\rn)$ and also \eqref{3.138}.
\\[0.1cm]
{\em Step 3.} Let $1<p_0 <p<\infty$ and $0<q\le \infty$. Then one has by \eqref{3.116} with $\frac{1}{p} = \frac{1-\theta}{p_0}$
\begin{\eq}    \label{3.146}
\| f \, | \Fa{}^{l- \frac{n}{p'}}_{p,q} (\rn) \| \le c\, \| f \, | \Ba{}^{l- \frac{n}{p'_0}}_{p_0, p_0} (\rn) \|^{1-\theta} \, \| f \, | \Ca{}^{l-n} (\rn)\|^\theta.
\end{\eq}
Inserting \eqref{3.143} with $p_0$ in place of $p$ in \eqref{3.146} one obtains \eqref{3.137}, \eqref{3.136} and the right-hand side of \eqref{3.138}.
The left-hand side is again covered by \eqref{3.145} and Theorem \ref{T3.19}.
\end{proof}

\begin{remark}   \label{R3.22}
Some special cases of \eqref{3.136} are already known. They have some history. In particular if $p= \frac{n}{n-l}$ then $\frac{n}{p'} =l$ and
\begin{\eq}   \label{3.147}
\| f \, | L_{\frac{n}{n-l}} (\rn) \| \le c \, \| f \, | \Wa{}^l_1 (\rn) \|^{\frac{n-l}{n}} \, \| f \, | \Ca{}^{l-n} (\rn) \|^{\frac{l}{n}},
\end{\eq}
$n>l\in \nat$. With $l=1$ one obtains
\begin{\eq}   \label{3.148}
\| f \, | L_{\frac{n}{n-1}} (\rn) \| \le c \, \| f \, | \Wa{}^1_1 (\rn) \|^{\frac{n-1}{n}} \, \| f \, | \Ca{}^{1-n} (\rn) \|^{\frac{1}{n}}
\end{\eq}
and if, in addition, $n=2$, then one has
\begin{\eq}   \label{3.149}
\| f \, | L_2 (\real^2) \| \le c \, \| f \, | \Wa{}^1_1 (\real^2) \|^{1/2} \, \| f \, | \Ca{}^{-1} (\real^2) \|^{1/2}.
\end{\eq}
Assertions of this type (with $l=1$) go back to \cite{CMO98} and have been generalized in \cite{CDPX99, CDDD03} including also some Besov spaces. 
Further references and discussions may be found in \cite[Section 4.1, pp.\,131--133]{T13}.
\end{remark}

\end{document}